\documentclass[a4paper,11pt]{article}
\usepackage[latin1]{inputenc}
\usepackage[english]{babel}
\usepackage{amsmath}
\usepackage{amsfonts}
\usepackage{amssymb}
\usepackage{epsfig}
\usepackage{amsopn}
\usepackage{amsthm}
\usepackage{color}
\usepackage{graphicx}
\usepackage{enumerate}
\newtheorem{theorem}{Theorem}[section]
\newtheorem{corollary}[theorem]{Corollary}
\newtheorem{lemma}[theorem]{Lemma}
\newtheorem{proposition}[theorem]{Proposition}

\newtheorem*{theorem*}{Theorem}
\newtheorem*{lemma*}{Lemma}
\newtheorem*{remark*}{Remark}
\newtheorem*{definition*}{Definition}
\newtheorem*{proposition*}{Proposition}
\newtheorem*{corollary*}{Corollary}
\numberwithin{equation}{section}
\newcommand{\real}{\mathbb{R}}

\let\ced=\c         

\def\a{\alpha}
\def\b{\beta}

\def\e{\varepsilon}        

\def\cl{{\cal L}}

\newcommand{\RR}{\mathbb{R}}

\def\qed{\,\unskip\kern 6pt \penalty 500
\raise -2pt\hbox{\vrule \vbox to8pt{\hrule width 6pt
\vfill\hrule}\vrule}\par}
\definecolor{darkblue}{rgb}{0.05, .05, .65}
\definecolor{darkgreen}{rgb}{0.1, .65, .1}
\definecolor{darkred}{rgb}{0.8,0,0}
\newcommand{\beqn}{\begin{equation}}
\newcommand{\eeqn}{\end{equation}}
\newcommand{\bear}{\begin{eqnarray}}
\newcommand{\eear}{\end{eqnarray}}
\newcommand{\bean}{\begin{eqnarray*}}
\newcommand{\eean}{\end{eqnarray*}}
\begin{document}

\title{\huge \bf Eternal solutions to a singular diffusion equation with critical gradient absorption\footnote{Partially supported by Laboratoire Europ\'een Associ\'e CNRS Franco-Roumain MathMode Math\'ematiques \& Mod\'elisation and ANR project CBDif-Fr ANR-08-BLAN-0333-01}}

\author{
\Large Razvan Gabriel Iagar\,\footnote{Departamento de An\'alisis
Matem\'atico, Univ. de Valencia, Dr. Moliner 50, 46100, Burjassot
(Valencia), Spain, \textit{e-mail:}
razvan.iagar@uv.es},\footnote{Institute of Mathematics of the
Romanian Academy, P.O. Box 1-764, RO-014700, Bucharest, Romania.}
\\[4pt] \Large Philippe Lauren\c cot\,\footnote{Institut de
Math\'ematiques de Toulouse, CNRS UMR~5219, Universit\'e de
Toulouse, F--31062 Toulouse Cedex 9, France. \textit{e-mail:}
Philippe.Laurencot@math.univ-toulouse.fr}\\ [4pt] }
\date{\today}
\maketitle

\begin{abstract}
The existence of nonnegative radially symmetric eternal solutions of
exponential self-similar type $u(t,x)=e^{-p\beta t/(2-p)}
f_\beta(|x|e^{-\beta t};\beta)$ is investigated for the singular
diffusion equation with critical gradient absorption
\begin{equation*}
\partial_{t} u-\Delta_{p} u+|\nabla u|^{p/2}=0 \quad \;\;\hbox{in}\;\; (0,\infty)\times\real^N
\end{equation*}
where $2N/(N+1)<p<2$. Such solutions are shown to exist only if the
parameter $\beta$ ranges in a bounded interval $(0,\beta_*]$ which
is in sharp contrast with well-known singular diffusion
equations such as $\partial_{t}\phi-\Delta_{p} \phi=0$ when
$p=2N/(N+1)$ or the porous medium equation
$\partial_{t}\phi-\Delta\phi^m=0$ when $m=(N-2)/N$. Moreover, the
profile $f(r;\beta)$ decays to zero as $r\to\infty$ in a faster way
for $\beta=\beta_*$ than for $\beta\in (0,\beta_*)$ but the
algebraic leading order is the same in both cases. In fact, for
large $r$, $f(r;\beta_*)$ decays as $r^{-p/(2-p)}$ while
$f(r;\beta)$ behaves as $(\log r)^{2/(2-p)} r^{-p/(2-p)}$ when
$\beta\in (0,\beta_*)$.
\end{abstract}

\vspace{2.0 cm}


\medskip


\newpage

\section{Introduction}\label{seint}

A commonly observed feature of nonnegative solutions to diffusion equations in the whole space $\real^N$ is their decay to zero as time increases to infinity. This convergence to zero takes place at different speeds depending on the equation under consideration (and also possibly on the initial data) and three different behaviours are usually observed. The most frequently met are algebraic decay to zero and finite time extinction. Roughly speaking, in the former, the $L^\infty$-norm of the solution at time $t>0$ decays as $t^ {-\alpha}$ for some positive parameter $\alpha$ depending on the equation and possibly on the integrability or decay properties of the initial data. In the latter, the solution is driven to zero in finite time and vanishes identically afterwards. Algebraic decay is well-known for the heat equation $\partial_t u -\Delta u = 0$ in $(0,\infty)\times\real^N$ and its nonlinear counterparts, the porous medium equation
\begin{equation}
\partial_t u -\Delta u^m = 0 \;\;\mbox{ in }\;\; (0,\infty)\times\real^N\,, \label{zpme}
\end{equation}
for $m>m_c:=(N-2)_+/N$ and the $p$-Laplacian equation
\begin{equation}
\partial_t u -\Delta_p u = 0 \;\;\mbox{ in }\;\; (0,\infty)\times\real^N\,, \label{zplap}
\end{equation}
for $p>p_c:=2N/(N+1)$. Finite time extinction is a more singular phenomenon and is already well-known for \eqref{zpme} when $m\in (0,m_c)$ and for \eqref{zplap} when $p\in (1,p_c)$, see \cite{VazquezBook, VazquezSmoothing} and the references therein. The above description reveals that, for the aforementioned examples, one value of the parameter is excluded, namely $m=m_c$ for \eqref{zpme} and $p=p_c$ for \eqref{zplap}. For these choices of the parameters $m$ or $p$, the convergence to zero is expected to be faster than any negative power of time without reaching zero in finite time. Exponential decay is then rather natural to be observed in these borderline cases though proving that it is indeed the case is far from being obvious, see \cite{GPV00}  for \eqref{zpme} with $m=m_c$ and \cite[Proposition~3.3]{IL1} for \eqref{zplap} with $p=p_c$. A difficult question is then to figure out which exponential decay rates are allowed or not, a characteristic property of critical exponents being the complexity of the possible behaviours. For instance, for the porous medium equation \eqref{zpme} with $m=m_c$, explicit self-similar solutions are available showing that, given any $a>0$, there is at least one solution with $L^\infty$-norm decaying exactly as $e^ {-at}$ as $t\to\infty$ \cite[Section~5.6.1]{VazquezSmoothing}. However, as shown in \cite{GPV00}, there are solutions decaying with a superexponential rate $e^ {-C t^ {N/(N-2)}}$. These results have a direct counterpart for the $p$-Laplacian equation \eqref{zplap} owing to the connection between radially symmetric solutions of the two equations established in \cite{ISV}.

A similar dichotomy has also been observed and thoroughly investigated for diffusion equations with absorption such as
\begin{equation}
\partial_t u -\Delta u^m + u^q =  0 \;\;\mbox{ in }\;\; (0,\infty)\times\real^N\,, \quad m>m_c\,,\label{zpmab}
\end{equation}
and
\begin{equation}
\partial_t u -\Delta_p u + u^q = 0 \;\;\mbox{ in }\;\; (0,\infty)\times\real^N\,, \quad p>p_c\,,\label{zplab}
\end{equation}
see \cite{GV03,VazquezSurvey} and the references therein. For these equations, algebraic decay takes place for $q>1$ while it readily follows from the comparison principle that there is finite time extinction when $q\in (0,1)$. More recently, diffusion equations with gradient absorption such as
\begin{equation}
\partial_t u -\Delta u^m + |\nabla u|^q = 0 \;\;\mbox{ in }\;\; (0,\infty)\times\real^N\,, \quad m>m_c\,,\label{zpmgab}
\end{equation}
and
\begin{equation}
\partial_t u -\Delta_p u + |\nabla u|^q = 0 \;\;\mbox{ in }\;\; (0,\infty)\times\real^N\,, \quad p>p_c\,,\label{zplgab}
\end{equation}
have been studied and algebraic decay have been obtained for
\eqref{zpmgab} when $(m,q)\in (m_c,1)\times (1,(2+mN)/(N+1))$
\cite{SW07} and $(m,q)\in (1,2)\times (1,2)$, $m<q$, \cite{ATU04}
and for \eqref{zplgab} when $(p,q)\in [2,\infty)\times (1,\infty)$
and $(p,q)\in (p_c,2)\times (p/2,\infty)$, see
\cite{ATU04,BL08,IL1,La08} and the references therein. Extinction in
finite time has also been established for \eqref{zplgab} when $p\in
(1,2]$ and $q\in (0,p/2)$ \cite{BLS02, BLSS02, IL1} with the
interesting novelty that the exponent $q$ below which the extinction
phenomenon takes place depends on the diffusion. In the borderline
case $q=1$ for \eqref{zpmab} and \eqref{zplab}) and $q=p/2$ for
\eqref{zplgab}, the situation seems to differ from that encountered
for the diffusion equations \eqref{zpme} and \eqref{zplap} as there
seems to be more constraints on the possible exponential decays.
Indeed, for \eqref{zpmab} and \eqref{zplab} with $q=1$,  a
straightforward application of the comparison principle guarantees
that the $L^\infty$-norm of the solution at time $t>0$ is bounded
from above by $e^ {-t}$ while a direct computation shows that the
$L^1$-norm of the solution decays exactly as $e^ {-t}$ for large
times. These two facts seem to indicate that arbitrary large
exponential decays are excluded. As for \eqref{zplgab} with $p\in
(p_c,2)$ and $q=p/2$, we proved in \cite[Theorem~1.2 \&
Proposition~5.2]{IL1} that, for initial data $u_0$ decaying
sufficiently rapidly at infinity, there are two positive constants
$C_1(u_0)>C_2(u_0)>0$ such that $e^ {-C_1(u_0) t} \le
\|u(t)\|_\infty \le e^{-C_2(u_0) t}$ for $t\ge 1$. Owing to the
dependence of the constants on $u_0$, we cannot deduce from this
result that only some exponential decay rates are admissible for
solutions to \eqref{zplgab} with $p\in (p_c,2)$ and $q=p/2$. The
purpose of this work is to go one step further in that direction by
studying the existence of self-similar solutions to this equation of
the form
\begin{equation}
u(t,x) = e^{-\alpha t}\ f\left( |x|e^{-\beta t} \right)\,, \quad (t,x)\in (0,\infty)\times\real^N\,, \label{self3}
\end{equation}
and to find out whether there are positive values of $\alpha$ and $\beta$ for which there are nonnegative and integrable solutions. As already mentioned, for \eqref{zpme} with $m=m_c$ and \eqref{zplap} with $p=p_c$, such solutions exist for any $\alpha>0$ with a specific value of $\beta$ depending on $\alpha$ and $N$.  In contrast, we will show in this paper that, for \eqref{zplgab} with $p\in (p_c,2)$ and $q=p/2$, there is a maximal decay rate $\alpha_*>0$ such that nonnegative and integrable solutions of the form \eqref{self3} only exist for $\alpha\in (0,\alpha_*]$, the corresponding profile $f$ having different properties for $\alpha\in (0,\alpha_*)$ and $\alpha=\alpha_*$.

We thus focus on the study of the existence and properties of solutions of the form \eqref{self3} to the
following singular diffusion equation
\begin{equation}\label{eq1}
\partial_{t}u-\Delta_{p}u+|\nabla u|^{p/2}=0, \quad
(t,x)\in(0,\infty)\times\real^N,
\end{equation}
where
\begin{equation}
p_c = \frac{2N}{N+1} < p <2\,. \label{zexp}
\end{equation}

Inserting the ansatz \eqref{self3} in \eqref{eq1} and setting $r=|x|e^{-\b t}$, we obtain that $\a$ and $\b$ shall satisfy
\begin{equation}
\a=\mu \b \,, \quad \mu := \frac{p}{2-p}\,, \label{zab}
\end{equation}
and the profile $f$ solves the differential equation
\begin{equation}\label{ODE1}
(|f'|^{p-2}f')'(r)+\frac{N-1}{r}(|f'|^{p-2}f')(r)+\a f(r)+\b
rf'(r)-|f'(r)|^{p/2}=0,
\end{equation}
with $f'(0)=0$. Next, it is straightforward to check that, if $f$ solves \eqref{ODE1} with $f'(0)=0$, then so does $f_\lambda : r\longmapsto \lambda^\mu\ f(\lambda r)$ for any $\lambda>0$ with $f_\lambda'(0)=0$ and $f_\lambda(0)=\lambda^\mu\ f(0)$.
Thanks to this scaling invariance and \eqref{zab}, we can restrict the analysis to the following problem
\begin{equation}\label{ODE2}
\left\{\begin{array}{l}
\displaystyle{(|f'|^{p-2}f')'(r)+\frac{N-1}{r}(|f'|^{p-2}f')(r)+\b(\mu
f(r)+rf'(r))-|f'(r)|^{p/2}=0},\\
 \\
f(0)=1, \ f'(0)=0,\end{array}\right.
\end{equation}
where $\mu=p/(2-p)>N$ by \eqref{zexp}. The main result of this paper uncovers a threshold value of the parameter $\beta$ below which \eqref{ODE2} has a positive solution defined on $[0,\infty)$ and identifies the behaviour of the corresponding solution as $r\to\infty$.

\begin{theorem}\label{th1}
There exists $\b_{*}>0$ such that, for any $\b\in(0,\b_{*}]$, there is a positive solution $f(\cdot;\b)\in C^1([0,\infty))$
to \eqref{ODE2} which satisfies:

\begin{itemize}
\item[(i)] If $\b=\b_{*}$, then $r^{\mu}f(r;\b_{*})\to w^{*}$ as
$r\to\infty$, where
\begin{equation}
w^*:=\frac{(\mu-N)^{2/(2-p)}}{\mu}\,. \label{zw}
\end{equation}

\item[(ii)] If $\b\in(0,\b_{*})$, then $r^{\mu}f(r;\b)\sim (K_{\infty}(\b)\log r)^{\mu+1}$ as $r\to\infty$, where $K_\infty(\b):=\mu^{p/2}/((\mu+1)\b)$.
\end{itemize}

In addition, for $\beta\in (0,\beta_*]$ and $t_0\in\real$, the function
\begin{equation*}
U_{\b,t_0}(t,x)=e^{-\mu \b (t+t_0)}f(|x|e^{-\b (t+t_0)};\b), \quad
(t,x)\in\real\times\real^N,
\end{equation*}
is a nonnegative and integrable self-similar solution to \eqref{eq1}.
\end{theorem}

We actually also prove that, if $\beta>\beta_*$, the initial value problem \eqref{ODE2} has a maximal solution $f(.;\beta)$ which is positive on $[0,R(\beta))$ for some $R(\beta)\in (0,\infty)$, vanishes at $R(\beta)$, and is negative in a right neighborhood of $R(\beta)$. Our study thus shows that, at least for nonnegative self-similar solutions, the temporal decay rate cannot exceed $e^ {-\beta_* t}$, which is in sharp contrast with what is known for \eqref{zpme} with $m=m_c$ and \eqref{zplap} with $p=p_c$.

Let us next point out that \eqref{ODE2} has several unusual features
compared to other ordinary differential equations associated to the
analysis of radially symmetric self-similar solutions for parabolic
equations, see \cite{BPT86, CQW03, IL2, Pe042, VazquezSmoothing} and
the references therein. First, the so-called ``shooting'' parameter
$\beta$ is here in the equation and not in the initial condition as
usual,which generates an additional term and thus additional
difficulties in the study of the variation $\partial_\beta
f(\cdot;\beta)$ of $f(\cdot;\beta)$ with respect to $\beta$. Next,
it is clear from Theorem~\ref{th1} that, though the decay of
$f(\cdot;\beta)$ as $r\to\infty$ is slower for $\beta\in
(0,\beta_*)$ than for $\beta=\beta_*$, the algebraic leading order
$r^ {-\mu}$ is the same and this tiny difference involving only a
logarithmic term complicates the analysis and requires finer
techniques. Indeed, in the aforementioned references, the fast
decaying orbit and the slow decaying orbits have different algebraic
rates.

An interesting byproduct of our analysis is that the self-similar solutions we construct in Theorem~\ref{th1} are actually \emph{eternal solutions}, that is, solutions defined for all times $t\in\real$. Since parabolic equations enjoy smoothing effects, the availability of such solutions is a rather casual phenomenon for such equations and might be observed only for very specific equations. In particular, for the two basic nonlinear diffusion equations \eqref{zpme} and \eqref{zplap}, there exist explicit one-parameter families of \emph{eternal solutions} of self-similar exponential type only when $m=m_c$ \cite{VazquezSmoothing} and $p=p_c$ \cite{ISV}, respectively. Also, if $N=2$, \emph{eternal solutions} are available for the logarithmic diffusion equation $\partial_t u - \Delta\log{u} = 0$ in $(0,\infty)\times \real^2$ which is related to the two-dimensional Ricci flow \cite{DS06}.

\medskip

Let us now describe the strategy of the proof of Theorem~\ref{th1}. Section~\ref{sebpf} is devoted to the local well-posedness of \eqref{ODE2} along with properties of the solution $f(\cdot;\beta)$ including a fine analysis of the behavior as $r\to 0$. In Section~\ref{semf}, we investigate the monotonicity properties of $r\mapsto r^{-\mu} f(\cdot;\beta)$ and divide the range of $\beta$  into three disjoint subsets $A$, $B$, and $C$ according to the expected behavior of $f(\cdot;\beta)$. In particular, global positive solutions to \eqref{ODE2} correspond to $\beta\in B\cup C$. With the aim of proving Theorem~\ref{th1}, a refined study of the sets $B$ and $C$ is required and relies on an intricate change of both variable and unknown function which is performed in Section~\ref{seaf} and allows us to reduce \eqref{ODE2} to a first-order differential equation. A careful study of this new equation then gives the precise behavior of $f(\cdot;\beta)$ as $r\to\infty$ by a delicate construction of suitable subsolutions and supersolutions. Of course, it depends upon whether $\beta$ belongs to $C$ (Section~\ref{seabC}) or $B$ (Section~\ref{seabB}). The latter enables us to show that $B$ is reduced to a single point.

\section{Basic properties of $f(\cdot;\b)$}\label{sebpf}

Fix $\beta>0$. Introducing $g:=-|f'|^{p-2}f'$, we observe
that \eqref{ODE2} also reads
\begin{equation}\label{ODE3}
\left\{\begin{array}{l}
f'(r)=-(|g|^{(2-p)/(p-1)}g)(r),\\
 \\
\displaystyle{ g'(r)+\frac{N-1}{r}g(r)} = \b(\mu
f(r)-r(|g|^{(2-p)/(p-1)}g)(r))-|g(r)|^{p/2(p-1)},\\
 \\
f(0)=1, \ g(0)=0.
\end{array}\right.
\end{equation}
Since $p\in(1,2)$, we have $p/2(p-1)>1$ and
$1+(2-p)/(p-1)=1/(p-1)>0$, and there is a unique maximal solution
$(f(\cdot;\b),g(\cdot;\b))$ to \eqref{ODE3}, which is $C^1$-smooth.

Let us define
\begin{equation*}
R(\b):=\inf\{r>0: \ f(r;\b)=0\}>0,
\end{equation*}
the positivity of $R(\b)$ being a straightforward consequence of the
continuity of $f(\cdot;\b)$. We begin with some basic properties of
$f(\cdot;\b)$. In the proofs of the following results we write
$f(r)=f(r;\b)$ and $g(r)=g(r;\beta)$, omitting the dependence on $\b$ to lighten notation.

\begin{lemma}\label{lemma.prop}
Let $\b>0$. We have $-(\b\mu)^{2/p}\leq f'(r;\b)<0$ for any
$r\in(0,R(\b))$. Moreover, if $R(\b)=\infty$, then
$$
\lim\limits_{r\to\infty}f(r;\b)=\lim\limits_{r\to\infty}f'(r;\b)=0.
$$
\end{lemma}

\begin{proof}
Let $g=-|f'|^{p-2}f'$. From \eqref{ODE3}, it follows that $g(0)=0$
and $g'(0)=\b\mu/N>0$, hence there is $\delta>0$ such that $f'(r)<0$
for $r\in(0,\delta)$. Set $r_0:=\inf\{r\in(0,R(\b)):f'(r)=0\}$
and assume for contradiction that $r_0<R(\b)$. Then, on the one hand,
$g(r_0)=f'(r_0)=0$ and we deduce from \eqref{ODE3} that
$g'(r_0)=\b\mu f(r_0)>0$. On the other hand, $g(r)>0=g(r_0)$ for
$r\in(0,r_0)$, whence $g'(r_0)\leq0$, which is a contradiction.
Consequently, $r_0\geq R(\b)$ and $f'<0$ in $(0,R(\b))$.

Consider next $R\in(0,R(\b))$ and let $r_m$ be a point of minimum of
$f'$ in $[0,R]$. Clearly, $r_m\ne 0$ and either $r_m\in(0,R)$ and
$f''(r_m)=0$ or $r_m=R$ and $f''(r_m)\leq0$. In both cases it
follows from \eqref{ODE2} and the negativity of $f'$ that $\b\mu
f(r_m)\geq|f'(r_m)|^{p/2}$. Consequently, if $r\in[0,R]$,
\begin{equation*}
|f'(r)|\leq |f'(r_m)|\leq(\b\mu f(r_m))^{2/p}\leq(\b\mu
f(0))^{2/p}=(\b\mu)^{2/p}.
\end{equation*}
Since $R\in (0,R(\beta))$ is arbitrary, we conclude that $|f'(r)|\le (\b\mu)^{2/p}$ for $r\in (0,R(\beta))$.
Finally, if $R(\b)=\infty$, we define the following ``energy"
\begin{equation}\label{energy}
E(r):=\frac{p-1}{p}|f'(r)|^{p}+\frac{\mu\b}{2}f(r)^2, \quad r>0.
\end{equation}
Then, owing to \eqref{ODE2} and the negativity of $f'$, we have
\begin{equation}\label{energderiv}
E'(r)=-\frac{N-1}{r}|f'(r)|^p-\b r|f'(r)|^2-|f'(r)|^{(p+2)/2}\leq0.
\end{equation}
Then $f$ and $E$ are two nonnegative and nonincreasing functions, so
that there exist $l\ge 0$ and $l_{E}\ge0$ such that $f(r)\to l$ and $E(r)\to
l_E$ as $r\to\infty$. On the one hand, it follows from
\eqref{energy} that $f'(r)$ has also a limit $l'$ as $r\to\infty$.
On the other hand, \eqref{energderiv} ensures that $f'$ belongs to
$L^{(p+2)/2}(0,\infty)$. Combining these two facts implies that
$l'=0$, from which we also deduce that $g(r)\to0$ as $r\to\infty$.
We then infer from \eqref{ODE3} that $g'(r)\to\mu\b l$ as
$r\to\infty$, which implies that $l=0$ since $g(r)\to 0$ as $r\to\infty$.
\end{proof}

For further use, we need to analyze in detail the behavior of $f(\cdot;\b)$ near $r=0$.

\begin{lemma}\label{lemma.exp1}
For $\b>0$, we have
\begin{equation}\label{expan}
\begin{split}
f(r;\b)&=1-C_1\left(\frac{\b\mu}{N}\right)^{1/(p-1)}r^{p/(p-1)}+C_2\left(\frac{\b\mu}{N}\right)^{(4-p)/2(p-1)}r^{3p/2(p-1)}\\
&+C_3(\b-B_1)\left(\frac{\b\mu}{N}\right)^{(3-p)/(p-1)}r^{2p/(p-1)}+o(r^{2p/(p-1)})
\end{split}
\end{equation}
as $r\to 0$, where
\begin{equation*}
C_1:=\frac{p-1}{p}, \quad C_2:=\frac{4(p-1)}{3p((2N+1)p-2N)}, \quad
C_3:=\frac{p-1}{2p(2-p)(p+N(p-1))},
\end{equation*}
and $B_1$ is defined in \eqref{zz1} below.
\end{lemma}

\begin{proof}
Since $(|f'|^{p-2}f')'(0)=-\mu\b/N$, we have that
$(|f'|^{p-2}f')(r)=-\mu\b r/N+o(r)$ as $r\to 0$, hence, owing to the nonnegativity of $f'$,
\begin{equation}\label{expan1}
f'(r)=-\left(\frac{\mu\b r}{N}\right)^{1/(p-1)}+o(r^{1/(p-1)})
\end{equation}
and
\begin{equation}\label{expan2}
f(r)=1-\frac{p-1}{p}\left(\frac{\mu\b}{N}\right)^{1/(p-1)}r^{p/(p-1)}+o(r^{p/(p-1)}),
\end{equation}
in a first order approximation. Since \eqref{ODE2} also reads
\begin{equation}\label{ODE4}
\frac{d}{dr}\left[ r^{N-1}(|f'|^{p-2}f')(r) \right] = r^{N-1}\left(|f'(r)|^{p/2}-\b(rf'(r)+\mu
f(r))\right),
\end{equation}
we infer from \eqref{expan1} and \eqref{expan2} that, as $r\to0$,
\begin{equation*}
\frac{1}{r^{N-1}}\frac{d}{dr}\left[ r^{N-1}(|f'|^{p-2}f')(r)\right]=\left(\frac{\mu\b
r}{N}\right)^{p/2(p-1)}-\b\mu+o(r^{p/2(p-1)}).
\end{equation*}
Integrating once, we find
$$
|f'(r)|^{p-1}=\frac{\b\mu}{N}r-\frac{2(p-1)}{p(2N+1)-2N}\left(\frac{\mu\b}{N}\right)^{p/2(p-1)}r^{(3p-2)/2(p-1)}+o(r^{(3p-2)/2(p-1)}),
$$
whence
\begin{equation}\label{expan3}
\begin{split}
f'(r)&=-\left(\frac{\mu\b}{N}\right)^{1/(p-1)}r^{1/(p-1)}+\frac{2}{p(2N+1)-2N}\left(\frac{\mu\b}{N}\right)^{(4-p)/2(p-1)}r^{(p+2)/2(p-1)}\\&+o(r^{(p+2)/2(p-1)}).
\end{split}
\end{equation}
Integrating once more gives the second order approximation as $r\to 0$:
\begin{equation}\label{expan4}
\begin{split}
f(r)&=1-\frac{p-1}{p}\left(\frac{\mu\b}{N}\right)^{1/(p-1)}r^{p/(p-1)}\\
&+\frac{4(p-1)}{3p(p(2N+1)-2N)}\left(\frac{\mu\b}{N}\right)^{(4-p)/2(p-1)}r^{3p/2(p-1)} +o(r^{3p/2(p-1)}).
\end{split}
\end{equation}
We then repeat the same technical step, inserting \eqref{expan3} and
\eqref{expan4} into \eqref{ODE4} in order to get the third order
approximation. Skipping straightforward computations, we arrive at
\begin{equation*}
\begin{split}
\frac{d}{dr}\left(r^{N-1}|f'(r)|^{p-1}\right)&=\b\mu
r^{N-1}-\left(\frac{\mu\b}{N}\right)^{p/2(p-1)}r^{(p/2(p-1))+N-1}\\&-\left(\frac{\mu\b}{N}\right)^{1/(p-1)}\frac{\b-B_0}{2-p}r^{(p/(p-1))+N-1}+o(r^{(p/(p-1))+N-1}),
\end{split}
\end{equation*}
where $B_0:=p(2-p)/(p(2N+1)-2N)$. After integration, we obtain the expansion of $f'$ as $r\to 0$,
\begin{equation}\label{expan5}
\begin{split}
f'(r)&=-\left(\frac{\mu\b}{N}\right)^{1/(p-1)}r^{1/(p-1)}+\frac{2}{p(2N+1)-2N}\left(\frac{\mu\b}{N}\right)^{(4-p)/2(p-1)}r^{(p+2)/2(p-1)}\\
&+\left[\frac{\b-B_0}{(2-p)(p+N(p-1))}-\frac{2 B_0^2}{p^2(2-p)}\right]\left(\frac{\mu\b}{N}\right)^{(3-p)/(p-1)}r^{(p+1)/(p-1)}\\&+o(r^{(p+1)/(p-1)}).
\end{split}
\end{equation}
Setting
\begin{equation}
B_1:=B_0+\frac{2(p+N(p-1))}{p^2}\ B_0^2\,, \quad B_0 = \frac{p(2-p)}{p(2N+1)-2N}\,, \label{zz1}
\end{equation}
one more integration of \eqref{expan5} gives \eqref{expan} with
the claimed constants $C_1$, $C_2$, and $C_3$.
\end{proof}

We will also use the expansion of $\partial_{\b}f(r;\b)$ as $r\to 0$ which we state now.

\begin{lemma}\label{lemma.exp2}
For $\b>0$, we have
\begin{equation}\label{expan.beta}
\begin{array}{l}
\displaystyle{\partial_{\b}f(r;\b)=-\frac{1}{p}\left(\frac{\mu}{N}\right)^{1/(p-1)}\ \b^{(2-p)/(p-1)}\ r^{p/(p-1)} + o(r^{p/(p-1)})}\,, \\
 \\
\displaystyle{\partial_\beta f'(r;\beta) = - \frac{1}{p-1}\ \left( \frac{\mu}{N} \right)^{1/(p-1)}\ \b^{(2-p)/(p-1)}\ r^{1/(p-1)} + o(r^{1/(p-1)})}\,,
\end{array}
\end{equation}
as $r\to 0$.
\end{lemma}

Formally, we obtain the expansions \eqref{expan.beta} by
differentiating with respect to $\b$ in \eqref{expan}. The rigorous
proof starts from differentiating with respect to $\b$ in
\eqref{ODE4} and follow the same steps as the proof of Lemma~\ref{lemma.exp1}. We omit the details and refer to \cite[Lemma~2.2]{IL2} where a similar result is proved.

At the end of this section, we apply the gradient estimates proved
in \cite[Theorem 1.3]{IL1}, to relate the growth of $f(\cdot;\beta)$ and $f'(\cdot;\beta)$.

\begin{lemma}\label{lemma.grad}
Let $\b>0$ such that $R(\b)=\infty$. Then $f(\cdot;\b)$ satisfies
\begin{equation}\label{gradest}
|f'(r;\b)|\leq C_4 f(r;\b)^{2/p}, \qquad r\geq0,
\end{equation}
for some constant $C_4>0$ depending only on $N$ and $p$.
\end{lemma}

\begin{proof}
As in \cite[Lemma 2.3]{IL2}, it is easy to check that the function
$$
u(t,x)=e^{-\mu\b t}f(|x|e^{-\b t};\b), \quad
(t,x)\in[0,\infty)\times\real^N,
$$
is a viscosity solution to \eqref{eq1} in the sense of \cite[Definition 6.1]{IL1} with initial condition $x\mapsto f(|x|;\b)$ belonging to $W^{1,\infty}(\real^N)$ due to Lemma~\ref{lemma.prop}. Recall that, owing to the singular diffusion, the classical definition of viscosity solution cannot be used and has to be adapted, see \cite{IS,OS}. We can then apply the gradient estimates in \cite{IL1} and deduce from \cite[Theorem 1.3, (ii)]{IL1} that there exists a positive constant $C_4$ depending only on $N$ and $p$
such that
$$
\left| \nabla u^{(p-2)/p}(t,x) \right| \leq \frac{(2-p) C_4}{p}\ (1+t^{-1/p}), \quad (t,x)\in(0,\infty)\times\real^N.
$$
Expressing this estimate in terms of $f(\cdot;\b)$ we obtain
$$
e^{-(\mu+1)\b t}|f'(r;\b)|\leq C_4\ e^{-2\mu\b
t/p}f(r;\b)^{2/p}(1+t^{-1/p}), \quad
(t,r)\in(0,\infty)\times [0,\infty).
$$
Taking into account that $2\mu\b/p=(\mu+1)\b$ and setting $t=1$, we obtain
\eqref{gradest}.
\end{proof}

\section{Monotonicity of $r\mapsto r^{-\mu}f(r;\b)$}\label{semf}

Following a technique already used in previous papers \cite{CQW03,
Pe042, IL2}, we next introduce the function $w$ defined by
\begin{equation}\label{defw}
w(r;\b)=r^{\mu}f(r;\b), \quad r\in[0,R(\b)), \ \b>0.
\end{equation}
Since $f'(r;\b)\ne 0$ for $r\in (0,R(\beta)$ by Lemma~\ref{lemma.prop}, it follows from \eqref{ODE2} that $w=w(\cdot;\beta)$ solves the differential equation
\begin{equation}\label{newfunct}
\begin{split}
(p-1)r^2w''(r)&+(N-1-2\mu(p-1))rw'(r)+\mu(\mu-N)w(r)\\
&+|rw'(r)-\mu w(r)|^{2-p}\left(\b rw'(r)-|rw'(r)-\mu
w(r)|^{p/2}\right)=0.
\end{split}
\end{equation}
Setting $w_{\b}(\cdot;\b)=\partial_{\b}w(\cdot;\b)$, we
differentiate \eqref{newfunct} with respect to $\b$ to find
\begin{equation}\label{part1}
\begin{split}
(p-1)r^2w_{\b}''(r)&+(N-1-2\mu(p-1))rw_{\b}'(r)+\mu(\mu-N)w_{\b}(r)\\
&+(2-p)(|W|^{-p}W)(r)(\b rw'(r)-|W(r)|^{p/2})(rw_{\b}'(r)-\mu
w_{\b}(r))\\&-\frac{p}{2}(|W|^{-p/2}W)(r)(rw_{\b}'(r)-\mu
w_{\b}(r))+\b|W(r)|^{2-p}rw_{\b}'(r)\\&=-|W(r)|^{2-p}rw'(r),
\end{split}
\end{equation}
where $W(r):=rw'(r)-\mu w(r)$. Let us remark at this point that, as a
difference with respect to previous works \cite{CQW03, Pe042, IL2},
the linear equation \eqref{part1} solved by $w_{\b}$ is non-homogeneous, that is, it has a nonzero right-hand side $-|W(r)|^{2-p} r w'(r)$. We next differentiate \eqref{part1} with
respect to $r$ and multiply the resulting identity by $r$ to obtain
after straightforward transformations that
\begin{equation}\label{part2}
\begin{split}
&(p-1)r^2(rw')''(r)+(N-1-2\mu(p-1))r(rw')'(r)+\mu(\mu-N)rw'(r)\\
&+(2-p)|W(r)|^{-p}W(r)(\b rw'(r)-|W(r)|^{p/2})(r(rw')'(r)-\mu
rw'(r))\\&-\frac{p}{2}|W(r)|^{-p/2}W(r)(r(rw')'(r)-\mu
rw'(r))+\b|W(r)|^{2-p}r(rw')'(r)=0.
\end{split}
\end{equation}
Introducing the differential operator
\begin{equation}
\begin{split}
L_{\b}(z)&:=(p-1)r^2z''+(N-1-2\mu(p-1))rz'+\mu(\mu-N)z\\
&+(2-p)|W(r)|^{-p}W(r)(\b rw'(r)-|W(r)|^{p/2})(rz'-\mu z)\\
&-\frac{p}{2}|W(r)|^{-p/2}W(r)(rz'-\mu z)+\b|W(r)|^{2-p}rz',
\end{split}
\end{equation}
we infer from \eqref{part1} and \eqref{part2} that
\begin{equation}\label{comp1}
L_{\b}(\partial_{\b}w(\cdot;\b))(r)=-|W(r)|^{2-p}rw'(r), \quad
L_{\b}(rw'(r;\b))=0, \quad r\in (0,R(\beta)).
\end{equation}
Our next goal is to show that the dependence of $w(\cdot;\b)$ with
respect to $\b$ is decreasing. To this end, let us first recall the following
comparison principle:

\begin{lemma}\label{lemma.comp}
Let $\b>0$, $r_1\in(0,R(\b))$ and $r_{2}\in(r_1,R(\b))$, and assume
that $w'(\cdot;\b)>0$ in $[r_1,r_2]$. Then, any function $h\in
C^2([r_1,r_2])$ satisfying $h(r_1)=h(r_2)=0$ and $L_{\b}(h)\geq0$ in
$(r_1,r_2)$, has the property that $h\leq0$ in $(r_1,r_2)$.
\end{lemma}

\begin{proof}
Owing to \eqref{comp1} and the positivity assumption on $w'(\cdot;\beta)$, Lemma~\ref{lemma.comp} follows from the variant of the comparison principle proved in \cite[p.~48]{BNV}.
\end{proof}

Using this comparison principle, we are able to prove the main
monotonicity result with respect to the parameter $\b$.

\begin{proposition}\label{prop.monot}
Let $\b>0$. Assume that there exists $r_0\in(0,R(\b))$ such that
$w'(\cdot;\b)> 0$ in $(0,r_0)$. Then
\begin{equation*}
\partial_{\b}w(r;\b)<0 \quad \hbox{for} \ r\in(0,r_0].
\end{equation*}
\end{proposition}

\begin{proof}
Set $w:=w(\cdot;\beta)$ and $w_\beta := \partial_\beta w(\cdot;\beta)$. Using the expansion \eqref{expan.beta} of $\partial_\beta f(\cdot;\beta)$ as $r\to 0$, we find
$$
-w_{\b}(r)=-r^{\mu}\partial_{\b}f(r;\b)\sim\frac{1}{p}\left(\frac{\mu}{N}\right)^{1/(p-1)}\b^{(2-p)/(p-1)}r^{p/(p-1)}
$$
as $r\to 0$, so that $-w_{\b}>0$ in a right neighborhood of $r=0$.
Setting
$$
r_1:=\inf\{r\in(0,r_0): w_{\b}(r)=0\},
$$
we have $r_1>0$ and $w_{\b}<0$ in $(0,r_1)$. Assume for contradiction that
$r_1<r_0$. Then $w_{\b}(r_1)=0=w_\beta(0)$ and $-w_\beta$ attains its positive maximum at some point $r_m\in (0,r_1)$.
Fix $\e>0$ such that
$$
\e\sup\limits_{[0,r_1]}\{ rw'(r) \} \leq - \frac{w_\beta(r_m)}{2} =
\frac{1}{2}\sup\limits_{[0,r_1]}(-w_{\b}(r)).
$$
Define
$$
z_{\e}(r):=-w_{\b}(r)-\e rw'(r), \quad r\in[0,r_1].
$$
On the one hand, $z_{\e}(r_1)=-\e r_1w'(r_1)<0$ and it follows from \eqref{expan1}, \eqref{expan2}, and \eqref{expan.beta} that, as $r\to 0$,
\begin{equation}\label{part3}
\begin{split}
z_{\e}(r)&=-r^{\mu}\partial_{\b}f(r;\b)-\e r^{\mu}(rf'(r;\beta)+\mu f(r;\beta))
\\&=-r^{\mu}\left(\frac{1}{p}\left(\frac{\mu}{N}\right)^{1/(p-1)}\b^{(2-p)/(p-1)}r^{p/(p-1)}+\e\mu
-\frac{\e\mu}{p}\left(\frac{\b\mu}{N}\right)^{1/(p-1)}r^{p/(p-1)}\right.\\&\left.+o(r^{p/(p-1)})\right)
\sim-\e\mu r^{\mu}.
\end{split}
\end{equation}
We may then choose $\delta\in(0,r_m)$ small enough such that $z_{\e}(\delta)<0$. On the other hand, by the choice of $\e>0$, we have
$$
z_{\e}(r_m)\geq\sup\limits_{[0,r_1]}\{-w_{\b}(r)\}-\e\sup\limits_{[0,r_1]}\{rw'(r)\}\geq - \frac{w_{\b}(r_m)}{2}>0.
$$
Since $z_{\e}(\delta)<0<z_{\e}(r_m)$ and $z_{\e}(r_1)<0<z_{\e}(r_m)$,
there exist $r_2\in(\delta,r_m)$ and $r_3\in(r_m,r_1)$ such that
\begin{equation}\label{slide}
z_{\e}(r_2)=z_{\e}(r_3)=0, \quad z_{\e}(r)>0 \ \hbox{for} \
r\in(r_2,r_3).
\end{equation}
By \eqref{comp1} and the positivity of $w'(\cdot;\beta)$, we have $L_{\b}(z_{\e})>0$ in $(r_2,r_3)$. Thus,
Lemma~\ref{lemma.comp} implies that $z_{\e}\leq0$ in $(r_2,r_3)$,
which contradicts \eqref{slide}. Consequently, $r_1=r_0$ and
\begin{equation}\label{comp2}
\partial_{\b} w(r;\b)<0 \quad \hbox{for} \ r\in(0,r_0).
\end{equation}

It remains to check that $\partial_{\b}w(r_0;\b)<0$. To this end, introduce the
Wronskian
$$
D(r):=-w_{\b}(r)\ v'(r)+w_{\b}'(r)\ v(r), \quad r\in [0,R(\beta))\,,
$$
with $v(r):=rw'(r)$. Then
$$
D'(r)=w_{\b}''(r)\ v(r)-w_{\b}(r)\ v''(r).
$$
Since $L_\beta(z)$ also reads $L_{\b}(z)(r)=(p-1)r^2\ \left[ z''(r)+a_1(r)z'(r)+a_0(r)z(r) \right]$ for
suitable functions $a_1$ and $a_0$, it follows from \eqref{comp1} that
$$
-|W(r)|^{2-p}\ v(r)=(p-1)r^2(w_{\b}''+a_1w_{\b}'+a_0w_{\b})(r)
$$
(recall that $W(r)=rw'(r)-\mu w(r)$) and
$$
0=(p-1)r^2(v''+a_1\ v'+a_0\ v)(r).
$$
Using these equalities, we can express $D'$ in terms of $D$,
obtaining the following differential inequality for $D$:
\begin{equation*}
\begin{split}
D'(r)&=v(r)\left[-\frac{|W(r)|^{2-p}\ v(r)}{(p-1)r^2}-a_1(r)\ w_{\b}'(r)-a_0(r)\ w_{\b}(r)\right]-w_{\b}(r)(-a_1\ v'-a_0\ v)(r)\\
&=-\frac{|W(r)|^{2-p}\ v(r)^2}{(p-1)r^2}-\left(a_1\ w_{\b}'\ v+a_0\ w_{\b}\ v-a_1\ w_{\b}\ v-a_0\ w_{\b}\ v\right)(r)\\
&\leq a_1(r)\ (-w_{\b}'(r)\ v(r)+w_{\b}(r)\ v'(r))=-a_1(r)\ D(r),
\end{split}
\end{equation*}
Therefore, by integration we find that
\begin{equation}\label{diff.ineq}
D(r)\leq D(s)\exp\left(-\int_s^r a_{1}(\tau)\,d\tau\right), \quad 0<s<r\leq r_0\,.
\end{equation}

We next express $D$ in terms of $f:=f(\cdot;\beta)$ and $f_\beta:=\partial_\beta f(\cdot;\beta)$ with the aim of studying its
behavior as $r\to 0$. Since
\begin{equation*}
\begin{split}
&v(r)=rw'(r)=\mu r^{\mu}f(r)+r^{\mu+1}f'(r),\\
&v'(r)=\mu^2r^{\mu-1}f(r)+(2\mu+1)r^{\mu}f'(r)+r^{\mu+1}f''(r)
\end{split}
\end{equation*}
and
$$
w_{\b}(r)=r^{\mu}f_{\b}(r), \quad w_{\b}'(r)=\mu
r^{\mu-1}f_{\b}(r)+r^{\mu}f_{\b}'(r),
$$
we have by straightforward computations
$$
D(r)=r^{2\mu}\left[f_{\b}'(r)(rf'(r)+\mu
f(r))-f_{\b}(r)((\mu+1)f'(r)+rf''(r))\right].
$$
Using Lemma~\ref{lemma.exp2}, we have as $r\to 0$,
\begin{equation*}
\begin{split}
&f_{\b}(r)\sim-\frac{1}{p}\left(\frac{\mu}{N}\right)^{1/(p-1)}\b^{(2-p)/(p-1)}r^{p/(p-1)},\\
&f_{\b}'(r)\sim-\frac{1}{p-1}\left(\frac{\mu}{N}\right)^{1/(p-1)}\b^{(2-p)/(p-1)}r^{1/(p-1)}
\end{split}
\end{equation*}
and, taking into account that $rf'(r)+\mu f(r)\sim\mu f(r)$ and Lemma~\ref{lemma.exp1}, we
have as $r\to 0$,
\begin{equation*}
D(r)\sim r^{2\mu}\mu
f(r)f_{\b}'(r)\sim-\frac{\mu}{p-1}\left(\frac{\mu}{N}\right)^{1/(p-1)}\b^{(2-p)/(p-1)}r^{2\mu+(1/(p-1))}.
\end{equation*}
Consequently, $D(0)=0$ and there is some $\delta>0$ sufficiently
small such that $D(s)<0$ for any $s\in(0,\delta)$. From
\eqref{diff.ineq} we deduce that
\begin{equation}
D(r)<0 \;\;\;\mbox{ for all }\;\;\; r\in(0,R(\b)).\label{zz2}
\end{equation}

Fix now $s_0\in(0,r_0)$ and let $\psi$ be the solution to
$L_{\b}(\psi)=0$ in $(s_0,r_0)$ with initial condition $\psi(s_0)=0$,
$\psi'(s_0)=1$. As $v(s)=sw'(s)>0$ for all $s\in(s_0,r_0)$,
Sturm's oscillation theorem guarantees that $\psi>0$ in $(s_0,r_0]$.
We define
$$
\varphi(r):=-w_{\b}(r)+\frac{w_{\b}(s_0)}{v(s_0)}\ v(r)+\frac{D(s_0)}{v(s_0)}\ \psi(r)\,, \quad r\in [s_0,r_0]\,,
$$
and notice that $\varphi(s_0)=\varphi'(s_0)=0$. Moreover,
$L_{\b}(\varphi)=-L_{\b}(w_{\b})>0$ in $(s_0,r_0)$ by \eqref{comp1} and the positivity of $w'$.
In particular $L_{\b}(\varphi)(s_0)=(p-1)s_0^2\varphi''(s_0)>0$,
hence $\varphi''(s_0)>0$, which implies that $\varphi>0$ in a right
neighborhood of $s_0$. Then Lemma~\ref{lemma.comp} guarantees that $\varphi$ cannot vanish in $(s_0,r_0]$ and thus
$\varphi>0$ in $(s_0,r_0]$. In particular, owing to \eqref{zz2},
\begin{equation*}
-w_{\b}(r_0)>-\frac{w_{\b}(s_0)}{v(s_0)}\ v(r_0) - \frac{D(s_0)}{v(s_0)}\ \psi(r_0)>\frac{|w_{\b}(s_0)|}{v(s_0)}\ v(r_0)\geq 0,
\end{equation*}
which ends the proof.
\end{proof}

\noindent \textbf{Splitting into three sets.} Coming back to
$w(\cdot;\b)$ which solves \eqref{newfunct}, we first note that
\eqref{newfunct} has two constant solutions, the zero solution and
the solution
\begin{equation}\label{wstar}
w^*:=\frac{(\mu-N)^{2/(2-p)}}{\mu}.
\end{equation}
In addition, it follows from \eqref{defw} and Lemma~\ref{lemma.exp1} that, as $r\to0$,
\begin{equation}
w'(r;\beta)=r^{\mu-1}(rf'(r;\beta)+\mu f(r;\beta))\sim\mu r^{\mu-1}, \label{zz3}
\end{equation}
whence $w'(\cdot;\beta)>0$ in a right neighborhood of $r=0$. As in
\cite{CQW03, IL2, Pe042} we then split the range $(0,\infty)$ of $\b$ into three disjoint
sets:
\begin{equation*}
\begin{split}
&A:=\{\b>0: \hbox{there} \ \hbox{exists} \ R_{1}(\b)\in(0,R(\b)) \
\hbox{such that} \ w'(R_{1}(\b);\b)=0\},\\
&B:=\{\b>0: w'(\cdot;\b)>0 \ \hbox{in} \ (0,\infty), \
\lim\limits_{r\to\infty}w(r;\b)<\infty\},\\
&C:=\{\b>0: w'(\cdot;\b)>0 \ \hbox{in} \ (0,\infty), \
\lim\limits_{r\to\infty}w(r;\b)=\infty\}.
\end{split}
\end{equation*}
Since $w'(\cdot;\b)>0$ in a right neighborhood of $r=0$, we indeed
have that $A\cup B\cup C=(0,\infty)$. We will next show that $A$ and
$C$ are open intervals, so that $B$ is nonempty and closed. In a second
step we will prove that $B$ reduces to a single point, proving in
this way Theorem~\ref{th1}.

\subsection{Characterization of the set $A$}\label{secsA}

As in \cite{CQW03, IL2, Pe042}, the following characterization of $A$ is available:

\begin{lemma}\label{lemma.caractA}
Let $\b>0$. Then the following four assertions are equivalent:

\noindent (a) $\b\in A$.

\noindent (b) There is $R_{1}(\b)\in(0,R(\b))$ such that
$w'(\cdot;\b)>0$ in $(0,R_1(\b))$, $w'(\cdot;\b)<0$ in
$(R_1(\b),R(\b))$ and $w''(R_1(\b);\b)<0$.

\noindent (c) We have
\begin{equation}\label{boundA}
\sup\limits_{r\in[0,R(\b))}w(r;\b)<w^*,
\end{equation}
where $w^*$ is defined by \eqref{wstar}.

\noindent (d) $R(\b)<\infty$.
\end{lemma}

Before proving it, we recall a general analysis result proved in, e.g., \cite[Lemma~2.9]{IL2}.

\begin{lemma}\label{lemma.sequence}
Let $h$ be a nonnegative function in $C^1([0,\infty))$ such that
there is a sequence $(r_k)_{k\ge 1}$, $r_k\to\infty$ as
$k\to\infty$, for which $h(r_k)\longrightarrow 0$ as $k\to\infty$.
Then, there is a sequence $(\rho_k)_{k\ge 1}$, $\rho_k\to\infty$ as
$k\to\infty$, such that $h(\rho_k)\longrightarrow 0$ and $\rho_k
h'(\rho_k)\longrightarrow 0$ as $k\to\infty$.
\end{lemma}

\begin{proof}[Proof of Lemma~\ref{lemma.caractA}]
Consider first $\b\in A$. Recalling \eqref{zz3}, we have
$$
R_1(\b):=\inf\{r>0: w'(r;\b)=0\}\in (0,R(\b))
$$
according to the definition of $A$, and $w$ is such that
$w'(\cdot;\b)>0$ in $(0,R_1(\b))$, $w'(R_1(\b);\b)=0$, and
$w''(R_1(\b);\b)\leq0$. Assume for contradiction that
$w''(R_1(\b);\b)=0$. It then follows from \eqref{newfunct} that
$$
\mu(\mu-N)w(R_1(\b);\b)-(\mu w(R_1(\b);\b))^{2-p/2}=0,
$$
that is, $w(R_1(\b);\b)=w^*$. Since $w'(R_1(\b);\b)=0$ and $w^*$ is a constant solution of \eqref{newfunct}, the well-posedness
of \eqref{newfunct} implies that $w(\cdot;\b)\equiv w^*$ in
$[0,R(\b))$, which contradicts the fact that $w(0;\b)=0$.
Consequently, $w''(R_1(\b);\b)<0$ and $w'(\cdot;\beta)$ is negative in a
right neighborhood of $R_1(\b)$. We then define
$$
R_2(\b):=\inf\{r\in(R_1(\b),R(\b)): w'(r;\b)=0\},
$$
and notice that $w'(r;\b)<0$ for $r\in(R_1(\b),R_2(\b))$. Assume for
contradiction that $R_2(\b)<R(\b)$. Then $w'(R_2(\b);\b)=0$ and
$w''(R_2(\b);\b)\geq0$. Evaluating \eqref{newfunct} at $r=R_1(\b)$
and at $r=R_2(\b)$, we find
$$
\mu(\mu-N)w(R_1(\b);\b)-(\mu
w(R_1(\b);\b))^{2-p/2}=-(p-1)R_1(\b)^2w''(R_1(\b);\b)>0
$$
and
$$
\mu(\mu-N)w(R_2(\b);\b)-(\mu
w(R_2(\b);\b))^{2-p/2}=-(p-1)R_2(\b)^2w''(R_2(\b);\b) \leq 0,
$$
from which we deduce that
\begin{equation}\label{comp3}
w(R_2(\b);\b)\geq w^*>w(R_1(\b);\b).
\end{equation}
This inequality contradicts the fact that $w(\cdot;\beta)$ is decreasing in
$(R_1(\b),R_2(\b))$. Therefore, $R_2(\b)=R(\b)$ and we have proved
that (a) implies (b).

Assume now that (b) holds true. Then $R_1(\b)$ is clearly a point of
maximum of $w(\cdot;\b)$ in $(0,R(\b))$ and it follows from
\eqref{newfunct} and \eqref{comp3} that
$$
\sup\limits_{r\in(0,R(\b))}w(r;\b)\leq w(R_1(\b);\b)<w^*,
$$
and thus assertion~(c).

Now, if $\b>0$ is such that \eqref{boundA} holds true, let us assume
for contradiction that $w'(\cdot;\b)>0$ in $(0,R(\b))$. Then $w(r;\beta)>w(0;\beta)=0$ for $r\in (0,R(\beta)$ which implies that
$R(\b)=\infty$. Moreover,
$$
\lim\limits_{r\to\infty} w(r;\beta) = \lambda := \sup_{r\in [0,\infty)}\{ w(r;\beta \} \in (0,w^*)\,,
$$
the bounds on $\lambda$ following from the positivity of $w(\cdot;\beta)$ and \eqref{boundA}. In particular, $w'(\cdot;\beta)\in L^1(0,\infty)$ and there exists a sequence $(r_k)_{k\ge 1}$ of positive real numbers, $r_k\to\infty$,
such that $r_kw'(r_k;\b)\longrightarrow 0$ as $k\to\infty$. Using Lemma~\ref{lemma.sequence}, we may find a sequence $(\varrho_k)_{k\ge 1}$, $\varrho_k\to\infty$, such that
$$
\lim\limits_{k\to\infty}\varrho_kw'(\varrho_k;\b)=\lim\limits_{k\to\infty}\varrho_k^2w''(\varrho_k;\b)=0\,.
$$
Taking $r=\varrho_k$ in \eqref{newfunct} and passing to the limit as
$k\to\infty$, we obtain that
$\mu(\mu-N)\lambda=(\mu\lambda)^{2-p/2}$, whence
$\lambda\in\{0,w^*\}$. Since we already know that $\lambda\in (0,w^*)$, we arrive at a contradiction. Therefore,
$w'(\cdot;\b)$ vanishes at least once in $(0,R(\b))$, hence $\b\in A$.

Consider now $\b\in A$ and assume for contradiction that $R(\b)=\infty$.
Then we deduce from (b) that $w$ is decreasing in
$(R_1(\b),\infty)$, hence $w$ has a limit $l\geq 0$ as $r\to\infty$.
Repeating the previous argument based on
Lemma~\ref{lemma.sequence}, it follows that $l\in\{0,w^*\}$, whence
$w(r;\b)\longrightarrow 0$ as $r\to\infty$ by \eqref{boundA}. Since $p<2$, we infer from
\eqref{gradest} that
$$
\frac{r|f'(r;\b)|}{f(r;\b)}\leq
Crf(r;\b)^{(2-p)/p}=Cw(r;\b)^{(2-p)/p}\mathop{\longrightarrow}_{r\to\infty} 0.
$$
Consequently, there exists $r_{*}>R_1(\b)$ such that
$$
-\mu\leq\frac{rf'(r;\b)}{f(r;\b)}\leq0 \ \hbox{for} \ \hbox{any} \
r>r_{*},
$$
which implies that $w'(r;\b)=r^{\mu}(rf'(r;\b)+\mu f(r;\b))\geq0$
for $r>r_*$. This contradicts the fact that $w(r;\b)\to 0$ as
$r\to\infty$. Hence $R(\b)<\infty$ and assertion~(d) is proved.

Finally, if $R(\b)<\infty$, then $w(R(\b);\b)=0=w(0;\b)$, which
implies that $w(\cdot;\beta)$ has a maximum point in $(0,R(\b))$, hence $\b\in
A$, thereby proving that (d) implies (a).
\end{proof}

We are now ready to identify the set $A$.

\begin{proposition}\label{prop.zA}
The set $A$ is an open interval of the form $(\b^*,\infty)$ for some
$\b^*>0$.
\end{proposition}

\begin{proof}
For $\b>0$, we introduce the function $F(\cdot;\beta)$ defined by
$f(r;\b)=F(r\b^{1/p};\b)$ for $r\in[0,R(\b))$. Then, letting
$s=r\b^{1/p}$, we have $f'(r;\b)=\b^{1/p}F'(s;\b)$ and it follows
from \eqref{ODE2} that $F=F(\cdot;\beta)$ satisfies for $s\in (0,R(\b)\b^{1/p})$,
\begin{equation*}
\left\{\begin{array}{l}
\displaystyle{(|F'|^{p-2}F')'(s)+\frac{N-1}{s}(|F'|^{p-2}F')(s)+sF'(s)+\mu F(s)-\b^{-1/2}|F'(s)|^{p/2}=0,}\\
 \\
 F(0)=1, \ F'(0)=0.
 \end{array}\right.
\end{equation*}
The limit problem as $\b\to\infty$ reads
\begin{equation}\label{probinfty}
\left\{\begin{array}{l}
\displaystyle{(|h'|^{p-2}h')'(s)+\frac{N-1}{s}(|h'|^{p-2}h')(s)+sh'(s)+\mu h(s)=0,}\\
 \\
h(0)=1, \ h'(0)=0.
\end{array}\right.
\end{equation}
The limit problem \eqref{probinfty} is well-known and has already been thoroughly studied, see \cite[Theorem~2]{Pe042} or \cite[Proposition~2.11]{IL2} for instance. In particular, there is $S_0>0$ such that $h(S_0)=0$, $h'(S_0)<0$ and $h'(s)<0<h(s)$ for $s\in (0,S_0)$. By continuous dependence, a similar property is enjoyed by $F$ for $\beta$ large enough (with a possibly different point depending on $\beta$) from which we deduce that there is $\bar{\beta}>0$ large enough such that $(\bar{\beta},\infty)\subset A$.

It remains to show that $A$ is an open interval. It first readily follows from Lemma~\ref{lemma.caractA}~(b) and the continuous dependence with respect to $\beta$ that $A$ is open. Next, using once more Lemma~\ref{lemma.caractA}~(b), we infer from the implicit function theorem that the function $\beta\mapsto R_1(\beta)$ belongs to $C^1(A)$. Consequently, the function $m:\beta\mapsto w(R_1(\beta);\beta)$ belongs to $C^1(A)$ and it follows from Proposition~\ref{prop.monot} (with $r_0=R_1(\beta)$) and Lemma~\ref{lemma.caractA}~(b) that
$$
\frac{dm}{d\beta}(\beta) = w'(R_1(\beta);\beta)\ \frac{dR_1}{d\beta}(\beta) + \partial_\beta w(R_1(\beta);\beta) = \partial_\beta w(R_1(\beta);\beta) < 0\,, \quad \beta\in A\,.
$$
Recalling that $w(\cdot;\beta)$ reaches its maximum at $R_1(\beta)$ for $\beta\in A$, we have thus shown that
\begin{equation}
w(R_1(\beta_2);\beta_2) = \sup\limits_{r\in (0,R(\beta_2))} \{ w(r;\beta_2) \} < \sup\limits_{r\in (0,R(\beta_1))} \{ w(r;\beta_1) \} = w(R_1(\beta_1);\beta_1) < w^* \label{zz5}
\end{equation}
for $(\beta_1,\beta_2)\in A\times A$ satisfying $\beta_1<\beta_2$, the last inequality being a consequence of Lemma~\ref{lemma.caractA}~(c).

Consider now $\beta_1\in A$ and define $\beta_2:=\inf\{ \beta>\beta_1\ :\ \beta\not\in A \}$. Since $A$ is open, we have $\beta_2>\beta_1$ and $(\beta_1,\beta_2)\subset A$. Assume for contradiction that $\beta_2<\infty$. Since $A$ is open, this implies that $\beta_2\in B\cup C$ and in particular that $R(\beta_2)=\infty$ and $w'(r;\beta_2)>0$ for all $r>0$. Given any integer $k\ge 1$, continuous dependence then ensures that $w'(k;\beta)\longrightarrow w'(k;\beta_2)>0$ as $\beta\nearrow\beta_2$. Thus, there is $\delta_k>0$ such that $w'(k;\beta)>0$ for $\beta\in (\beta_2-\delta_k,\beta_2)\subset (\beta_1,\beta_2)$. Therefore, according to Lemma~\ref{lemma.caractA}~(b), $R_1(\beta)>k$ for $\beta\in (\beta_2-\delta_k,\beta_2)$ and thus
\begin{equation}
\lim\limits_{\beta\nearrow\beta_2} R_1(\beta)=\infty\,. \label{zz6}
\end{equation}
Now, for $r\in (0,\infty)$, we infer from \eqref{zz6} that $r\in (0,R_1(\beta))\subset (0,R(\beta))$ for $\beta<\beta_2$ close enough to $\beta_2$ which ensures that $w(r;\beta)\le w(R_1(\beta);\beta) = m(\beta)<m(\beta_1)<w^*$ by \eqref{zz5}. Since
$$
w(r;\beta_2) = \lim\limits_{\beta\nearrow\beta_2} w(r;\beta)
$$
by continuous dependence, we deduce that $w(r;\beta_2)\le m(\beta_1)<w^*$ for all $r>0$ which implies that $\beta_2\in A$ by Lemma~\ref{lemma.caractA}~(c) and a contradiction. We have thus established that $\beta_2=\infty$ from which Proposition~\ref{prop.zA} follows.
\end{proof}

\subsection{Characterization of the set $C$}\label{secsC}

We turn now our attention to the set $C$ and show that it is also
an open interval.

\begin{proposition}\label{prop.caractC}
\noindent (a) We have $\b\in C$ if and only if
\begin{equation}\label{caractC}
\sup\limits_{r\in(0,R(\b))}w(r;\b)>w^*.
\end{equation}

\noindent (b) The set $C$ is an open interval of the form
$(0,\b_{*})$ for some $\b_{*}>0$.
\end{proposition}

\begin{proof}
\noindent (a) If $\b\in C$, the inequality \eqref{caractC} is an
immediate consequence of the definition of $C$. Conversely, if
$\b>0$ such that \eqref{caractC} holds true, then $\b\in B\cup C$ by
Lemma~\ref{lemma.caractA}. Therefore, $w(\cdot;\b)$ is an increasing function
in $(0,\infty)$. If $w(\cdot;\b)$ is bounded, then it has a finite limit as
$r\to\infty$, and by standard arguments this limit has to be $w^*$,
contradicting \eqref{caractC}. Thus, $w(\cdot;\b)$ is unbounded,
whence $\b\in C$.

\noindent (b) We first show that $C$ is nonempty. Given $\b>0$, it
follows from Lemma~\ref{lemma.prop} that
$0>f'(r;\b)\geq-(\b\mu)^{2/p}$ for all $r\in(0,R(\b))$, whence
\begin{equation}\label{part4}
1-(\b\mu)^{2/p}r\leq f(r;\b)<1, \quad r\in(0,R(\b)).
\end{equation}
This inequality implies in particular that
$R(\b)\geq(\b\mu)^{-2/p}$. Thus, $(\b\mu)^{-2/p}/2$ belongs to
$(0,R(\b))$ and we evaluate the first part of the inequality
\eqref{part4} at this point, getting
$$
w\left(\frac{(\b\mu)^{-2/p}}{2};\b\right)\geq\left(\frac{1}{2(\b\mu)^{2/p}}\right)^{p/(2-p)}\left(1-\frac{(\b\mu)^{2/p}}{2(\b\mu)^{2/p}}\right)
=\left(\frac{1}{2\b\mu}\right)^{2/(2-p)}.
$$
Consequently,
\begin{equation}\label{part5}
\sup\limits_{r\in(0,R(\b))}w(r;\b)\geq (2\b\mu)^{-2/(2-p)}>w^*,
\end{equation}
for $\b$ small enough, hence $\b\in C$. The fact that $C$ is an open
interval follows directly from Proposition~\eqref{prop.caractC}~(a) and the monotonicity with respect to $\beta$ stated in
Proposition~\ref{prop.monot}.
\end{proof}

As a further consequence of Lemma~\ref{lemma.caractA} and Proposition~\ref{prop.caractC}, we may identify $B$ and the behavior of $w(r;\beta)$ as $r\to \infty$ for $\beta\in B$.

\begin{corollary}\label{cor.zz}
The set $B$ is the closed interval $B=[\beta_*,\beta^*]$. Moreover, if $\beta\in B$ then $w(r;\beta)\to w^*$ as $r\to\infty$.
\end{corollary}

\begin{proof}
The fact that $B=[\beta_*,\beta^*]$ readily follows from $A\cup B\cup C=(0,\infty)$, Proposition~\ref{prop.zA}, and Proposition~\ref{prop.caractC}. Next, according to the definition of $B$ and Proposition~\ref{prop.caractC}~(a), $w(\cdot;\beta)$ is increasing and bounded from above by $w^*$. Then
$$
\ell := \lim_{r\to\infty} w(r;\beta) = \sup_{r\in [0,\infty)}{\{ w(r;\beta) \}} \le w^*
$$
and, since $\beta\not\in A$, we infer from Lemma~\ref{lemma.caractA} that $\ell\ge w^*$, whence $\ell=w^*$.
\end{proof}

\section{An alternative formulation of \eqref{newfunct} when $\b\in B\cup C$}\label{seaf}

In this section we provide a deeper analysis of the differential
equation \eqref{newfunct}, which in the end will lead us to the
proof of Theorem~\ref{th1}. Consider $\b\in B\cup C$. Then $R(\beta)=\infty$,
$w'(r;\b)>0$ for all $r>0$, and
$$
\lim\limits_{r\to\infty}w(r;\beta)=\xi^*(\beta), \quad\mbox{ where }
\xi^*(\beta):=\left\{
\begin{array}{l}
w^* \quad \hbox{if} \ \b\in B,\\
\infty \quad \hbox{if} \ \b\in C.
\end{array}\right.
$$
Since $w(0;\beta)=0$, it follows that $w(\cdot;\beta)$ is a one-to-one mapping from
$[0,\infty)$ to $[0,\xi^*(\beta))$. Thus, we can define a new function
$\Phi(\cdot;\beta)$ by
\begin{equation}\label{def.Phi}
\Phi(\cdot;\beta):[0,\xi^*(\beta))\mapsto[0,\infty), \quad \Phi(w(r;\b);\beta)=rw'(r;\b), \
r\in[0,\infty).
\end{equation}
This change of function is very useful since it reduces the order of
\eqref{newfunct}. Indeed, observing that
\begin{equation*}
rw''(r;\b)+w'(r;\b)=\Phi'(w(r;\b);\beta)w'(r;\b), \quad
r^2w''(r;\b)=(\Phi(\Phi'-1))(w(r;\b);\beta)
\end{equation*}
and introducing the new independent variable $\xi:=w(r;\b)$,
\eqref{newfunct} reads
\begin{equation}\label{equation.Phi}
\begin{split}
(p-1)(\Phi\Phi')(\xi;\beta)&+(N-\mu p)\Phi(\xi;\beta)+\mu(\mu-N)\xi\\
&+|\Phi(\xi;\beta)-\mu\xi|^{2-p}\left(\b\Phi(\xi;\beta)-|\Phi(\xi;\beta)-\mu\xi|^{p/2}\right)=0
\end{split}
\end{equation}
for $\xi\in[0,\xi^*(\beta))$ with $\Phi(0;\beta)=0$. Note that we reduced
\eqref{newfunct} to a first-order differential equation. Also, since
$\b\in B\cup C$, it follows from \eqref{def.Phi} that
\begin{equation}
\Phi(\xi;\beta)>0 \quad\mbox{ for all }\quad \xi\in (0,\xi^*(\beta)).\label{zz4}
\end{equation}

\subsection{Behavior of $\Phi(\cdot;\beta)$ as $\xi\to 0$}\label{seabphx0}

\begin{lemma}\label{lemma.zPhi} For $\beta\in B\cup C$, we have $\Phi'(0;\beta)=\mu$ and, as $\xi\to 0$,
\begin{eqnarray}
\Phi'(\xi;\beta) & = & \mu-\frac{1}{p-1}\ \left( \frac{\mu\beta}{N} \right)^{1/(p-1)}\ \xi^{(2-p)/(p-1)}+o(\xi^{(2-p)/(p-1)}), \label{expan1.Phi}\\
\Phi(\xi;\beta) & = & \mu\xi-\left( \frac{\mu\beta}{N} \right)^{1/(p-1)}\ \xi^{1/(p-1)}+o(\xi^{1/(p-1)}). \label{expan2.Phi}
\end{eqnarray}
\end{lemma}

\begin{proof}
Set $f=f(\cdot;\beta)$, $w=w(\cdot;\beta)$, and $\Phi=\Phi(\cdot;\beta)$ to ease notations. According to
\eqref{expan1} and \eqref{expan2}, we have
$$
w(r)=r^{\mu}f(r)=r^{\mu}\left(1-\frac{p-1}{p}br^{p/(p-1)}+o(r^{p/(p-1)})\right) \quad\mbox{ with }\quad b:= \left( \frac{\mu\beta}{N} \right)^{1/(p-1)},
$$
and
\begin{equation*}
\begin{split}
rw'(r)&=r^{\mu+1}f'(r)+\mu
r^{\mu}f(r)=r^{\mu}\left(-br^{p/(p-1)}+o(r^{p/(p-1)})+\mu-\frac{(p-1)\mu}{p}br^{p/(p-1)}\right)\\
&=r^{\mu}\left(\mu-\frac{\mu}{p}br^{p/(p-1)}+o(r^{p/(p-1)})\right)
\end{split}
\end{equation*}
as $r\to 0$. We also have
\begin{equation*}
\begin{split}
rw'(r)-\mu w(r)&=r^{\mu}\left(\mu-\frac{\mu}{p}br^{p/(p-1)}-\mu+\frac{\mu(p-1)}{p}br^{p/(p-1)}+o(r^{p/(p-1)})\right)\\
&=-b r^{\mu/(p-1)}+o(r^{\mu/(p-1)})
\end{split}
\end{equation*}
as $r\to 0$. Inserting the previous expansions as $r\to 0$ in \eqref{newfunct}, we infer that
\begin{equation*}
\begin{split}
(p-1)r^2w''(r)&=\mu
r^{\mu}\left(-(N-1-2\mu(p-1))-\mu+N\right)+o(r^{\mu/(p-1)})\\&+\mu
r^{\mu/(p-1)}\left(\frac{b(N-1-2\mu(p-1))}{p}+\frac{(\mu-N)(p-1)b}{p} - \b
b^{2-p}\right)\\
&=2\frac{\mu(p-1)^2}{2-p}r^{\mu}-\mu
br^{\mu/(p-1)}\left(\frac{1}{p}+\frac{p-1}{2-p}\right)+o(r^{\mu/(p-1)}).
\end{split}
\end{equation*}
Then, as $r\to 0$,
\begin{equation*}
\begin{split}
\Phi'(w(r))&=\frac{rw''(r)+w'(r)}{w'(r)}=\frac{r^2w''(r)+rw'(r)}{rw'(r)}\\
&=\frac{\frac{2\mu(p-1)}{2-p}r^\mu-\frac{\mu
b}{p-1}\left(\frac{1}{p}+\frac{p-1}{2-p}\right)r^{\mu/(p-1)}+\mu
r^{\mu}-\frac{b\mu}{p}r^{\mu/(p-1)}+o(r^{\mu/(p-1)})}{\mu
r^{\mu}\left(1-\frac{b}{p}r^{p/(p-1)}+o(r^{p/(p-1)})\right)}\\
&=\frac{\left(1+\frac{2(p-1)}{2-p}\right)-br^{p/(p-1)}\left(\frac{1}{p}+\frac{1}{p(p-1)}+\frac{1}{2-p}\right)+o(r^{p/(p-1)})}{1-\frac{b}{p}r^{p/(p-1)}+o(r^{p/(p-1)})}\\
&=\left(\mu-\frac{b}{(p-1)(2-p)}r^{p/(p-1)}+o(r^{p/(p-1)})\right)\left(1+\frac{b}{p}r^{p/(p-1)}+o(r^{p/(p-1)})\right)\\
&=\mu-\frac{b}{p-1}r^{p/(p-1)}+o(r^{p/(p-1)}).
\end{split}
\end{equation*}
Since $w(r)=r^{\mu}+o(r^{\mu})$  as $r\to 0$, we end up with
$$
\Phi'(w(r))=\mu-\frac{b}{p-1}w(r)^{(2-p)/(p-1)}+o(w(r)^{(2-p)/(p-1)})
$$
as $r\to 0$, whence \eqref{expan1.Phi}. Integrating \eqref{expan1.Phi} gives \eqref{expan2.Phi}.
\end{proof}

Using these expansions as $\xi\to 0$, we are able to prove the
following upper bound.

\begin{lemma}\label{lemma.estPhi}
For $\beta\in B\cup C$, we have
\begin{equation}\label{estimate.Phi1}
0<\Phi(\xi;\beta)<\mu\xi \quad \hbox{for} \ \hbox{any} \ \xi\in(0,\xi^*(\beta)).
\end{equation}
\end{lemma}

\begin{proof}
It follows from \eqref{expan2.Phi} that there exists $\delta>0$ such
that $\Phi(\xi;\beta)<\mu\xi$ for $\xi\in(0,\delta)$. Setting
$$
\xi_0:=\inf\{\xi\in(0,\xi^*(\beta)): \Phi(\xi;\beta)=\mu\xi\},
$$
we have just shown that $\xi_0>0$ and $\Phi(\xi)<\mu\xi$ for
$\xi\in(0,\xi_0)$. Assume for contradiction that $\xi_0<\xi^*(\beta)$. Then
$\Phi(\xi_0;\beta)=\mu\xi_0$. Setting $\Phi_{l}(\xi)=\mu\xi$ for $\xi\in [0,\xi^*(\beta))$, it is easy
to check that $\Phi_{l}$ solves \eqref{equation.Phi}.
Since $\Phi(\xi_0)=\Phi_{l}(\xi_0)\neq0$ and both $\Phi$ and
$\Phi_l$ solve \eqref{equation.Phi}, we conclude that
$\Phi\equiv\Phi_l$, which contradicts the definition of $\xi_0$.
Consequently, $\xi_0=\xi^*(\beta)$ and \eqref{estimate.Phi1} holds true.
\end{proof}

\subsection{Monotonicity with respect to $\b$}\label{seamrb}

We have the
following ordering property.

\begin{lemma}\label{lemma.monotPhi}
Given $0<\b_1<\b_2$, we have $\Phi(\xi;\b_2)<\Phi(\xi;\b_1)$ for
$\xi\in(0,\min\{\xi^*(\beta_1),\xi^*(\beta_2)\})$.
\end{lemma}

\begin{proof}
To simplify notations, define $\Phi_i:=\Phi(\cdot;\b_i)$, $i=1,2$.
It follows from \eqref{expan2.Phi} that, as $\xi\to 0$,
\begin{equation*}
\begin{split}
\Phi_2(\xi)-\Phi_1(\xi)&=\mu\xi-\left(\frac{\mu\b_2}{N}\right)^{1/(p-1)}\xi^{1/(p-1)}-\mu\xi+\left(\frac{\mu\b_1}{N}\right)^{1/(p-1)}\xi^{1/(p-1)}+o(\xi^{1/(p-1)})\\
&=\left[\left(\frac{\mu\b_1}{N}\right)^{1/(p-1)}-\left(\frac{\mu\b_2}{N}\right)^{1/(p-1)}\right]\xi^{1/(p-1)}+o(\xi^{1/(p-1)}),
\end{split}
\end{equation*}
hence $\Phi_2(\xi)<\Phi_1(\xi)$ in a right neighborhood of $\xi=0$.
Introducing
$$
\xi_0:=\inf\{\xi\in(0,\min\{\xi^*(\beta_1),\xi^*(\beta_2)\})\ :\
\Phi_1(\xi)=\Phi_2(\xi)\},
$$
we have thus shown that $\xi_0>0$ and $\Phi_2(\xi)<\Phi_1(\xi)$ for
$\xi\in(0,\xi_0)$. Assume now for contradiction that
$\xi_0<\min\{\xi^*(\beta_1),\xi^*(\beta_2)\}$. Then $\Phi_1(\xi_0)=\Phi_2(\xi_0)$, $\Phi_1'(\xi_0)\leq\Phi_2'(\xi_0)$, and we infer from \eqref{equation.Phi} that
\begin{equation*}
\begin{split}
0&=(p-1)\Phi_1(\xi_0)\Phi_{1}'(\xi_0) + (N-\mu p)\Phi_1(\xi_0)+\mu(\mu-N)\xi_0\\
&+\b_1\Phi_1(\xi_0)|\Phi_1(\xi_0)-\mu\xi_0|^{2-p}-|\Phi_1(\xi_0)-\mu\xi_0|^{(4-p)/2}\\
&=(p-1)\left[(\Phi_1\Phi_1')(\xi_0)-(\Phi_2\Phi_2')(\xi_0)\right]+(p-1)(\Phi_2\Phi_2')(\xi_0)\\
&+ (N-\mu p) \Phi_2(\xi_0)+\mu(\mu-N) \xi_0+\b_2\Phi_2(\xi_0)|\Phi_2(\xi_0)-\mu\xi_0|^{2-p}\\
&+(\b_1-\b_2)\Phi_1(\xi_0)|\Phi_1(\xi_0)-\mu\xi_0|^{2-p}-|\Phi_2(\xi_0)-\mu\xi_0|^{(4-p)/2}\\
&=\Phi_1(\xi_0)\left[(p-1)(\Phi_1-\Phi_2)'(\xi_0)+(\b_1-\b_2)|\Phi_1(\xi_0) - \mu \xi_0|^{2-p}\right].
\end{split}
\end{equation*}
Since both terms in the right-hand side of the last equality above
are nonpositive and $0<\Phi_1(\xi_0)<\mu\xi_0$ by
\eqref{estimate.Phi1}, we end up with
$(\Phi_1-\Phi_2)'(\xi_0)=0=\b_1-\b_2$, and a contradiction. Consequently,
$\xi_0=\min\{\xi^*(\beta_1),\xi^*(\beta_2)\}$.
\end{proof}

With these preliminaries and general properties of $\Phi(\cdot;\beta)$, we are
now ready to separate the study in two cases, depending on
whether $\b\in B$ and $\b\in C$.

\subsection{Asymptotic behavior as $\b\in C$}\label{seabC}

For $\beta\in C$, the upper bound \eqref{estimate.Phi1} turns out to overestimate the growth of $\Phi(\cdot;\beta)$ for large values of $\xi$. A finer upper bound is shown in the next result which is also non-optimal as we shall see below but paves the way to the optimal growth rate established in Lemma~\ref{lemma.super}.

\begin{lemma}\label{lemma.boundC}
Consider $\b\in C$ and some positive constant $K$ such that
\begin{equation}\label{part6}
K\geq \max{\left\{ \frac{\mu^{p/2}}{\b} , \frac{p\mu-N}{\mu-N} \right\}}.
\end{equation}
Then
\begin{equation}\label{ineq.C}
\Phi(\xi;\beta)\leq K\xi^{p/2} \quad \hbox{for} \ \xi\in[0,\infty).
\end{equation}
\end{lemma}

\begin{proof}
Owing to \eqref{part6}, we have $K^{2/p}>\mu\b^{-2/p}$ and thus
$K^{(p-2)/p}\b^{-2/p}<K/\mu$. There is therefore some $\xi_{\b}>0$
such that
\begin{equation}\label{part7}
\max\left\{K^{(p-2)/p}\b^{-2/p},\frac{p\mu-N}{\mu(\mu-N)}\right\}\leq\xi_{\b}^{(2-p)/2}\leq\frac{K}{\mu}.
\end{equation}
We define $\Phi_{u}(\xi):=K\xi^{p/2}$ for $\xi\geq0$ and denote the
differential operator applied to $\Phi(\cdot;\beta)$ in \eqref{equation.Phi} by
$\cl$. Then, for $\xi\geq0$, we have
\begin{equation*}
\begin{split}
\cl\Phi_{u}(\xi)&=\frac{K^2 p(p-1)}{2}\xi^{p-1}+K (N-p\mu)\xi^{p/2}+\mu(\mu-N)\xi\\
&+|K\xi^{p/2}-\mu\xi|^{2-p}\left(K\b\xi^{p/2}-|K\xi^{p/2}-\mu\xi|^{p/2}\right)\\
&\geq
K (N-p\mu) \xi^{p/2}+\mu(\mu-N)\xi\\
&+\xi^{p/2}|K\xi^{p/2}-\mu\xi|^{2-p}\left(K\b-|K\xi^{(p-2)/2}-\mu|^{p/2}\right)\\
&\geq\mu(\mu-N)\xi-K(p\mu-N)\xi^{p/2}+\xi^{p/2}|K\xi^{p/2}-\mu\xi|^{2-p}\left(K\b-|K\xi^{(p-2)/2}-\mu|^{p/2}\right).
\end{split}
\end{equation*}
Now, since $\xi_{\b}^{(2-p)/2}\leq K/\mu$ by \eqref{part7},
we have:

\noindent $\bullet$ either $\xi\geq(K/\mu)^{2/(2-p)}$, hence
$\xi^{(2-p)/2}\geq K/\mu$, $\mu\geq K\xi^{(p-2)/2}$, and, owing to
\eqref{part6},
$$
\left|K\xi^{(p-2)/2}-\mu\right|^{p/2}=\left(\mu-K\xi^{(p-2)/2}\right)^{p/2}\leq\mu^{p/2}\leq
K\b.
$$

\noindent $\bullet$ or $\xi_{\b}\leq\xi\leq(K/\mu)^{2/(2-p)}$, hence
$\xi^{(2-p)/2}\leq K/\mu$ or $\mu\leq K\xi^{(p-2)/2}$, and, since
$p\in(1,2)$, we infer from the previous inequalities and
\eqref{part7} that
$$
\left|K\xi^{(p-2)/2}-\mu\right|^{p/2}=\left(K\xi^{(p-2)/2}-\mu\right)^{p/2}\leq\left(K\xi^{(p-2)/2}\right)^{p/2}\leq\left(K\xi_{\b}^{(p-2)/2}\right)^{p/2} \le K\b.
$$
Moreover, \eqref{part7} guarantees that for $\xi\geq\xi_{\b}$,
\begin{equation*}
\begin{split}
\mu(\mu-N)\xi-(p\mu-N)\xi^{p/2}&=\xi^{p/2}\left[\mu(\mu-N)\xi^{(2-p)/2}-(p\mu-N)\right]\\
&\geq\xi^{p/2}\left[\mu(\mu-N)\xi_{\b}^{(2-p)/2}-(p\mu-N)\right]\geq0.
\end{split}
\end{equation*}
Consequently,
$$
\cl\Phi_{u}(\xi)\geq0, \quad \hbox{for} \ \xi\in[\xi_{\b},\infty).
$$
Since
$$
\Phi_{u}(\xi_{\b})=K\xi_{\b}^{p/2}=K\xi_{\b}^{(p-2)/2}\xi_{\b}\geq\mu\xi_{\b}\geq\Phi(\xi_{\b};\beta)
$$
by \eqref{estimate.Phi1} and \eqref{part7}, the comparison principle ensures that $\Phi_{u}(\xi)\geq\Phi(\xi;\beta)$
for $\xi\geq\xi_{\b}$. In addition, if $\xi\in(0,\xi_{\b})$, we also deduce from \eqref{estimate.Phi1} and \eqref{part7} that
$$
\Phi(\xi;\beta)<\mu\xi^{p/2}\xi^{(2-p)/2}\leq\mu\xi^{p/2}\xi_{\b}^{(2-p)/2}\leq
K\xi^{p/2}=\Phi_{u}(\xi),
$$
which concludes the proof.
\end{proof}

We notice that, at a formal level, if $\Phi(\xi;\beta)\sim K\xi^{p/2}$ as
$\xi\to\infty$, then $rw'(r;\b)\sim Kw(r;\b)^{p/2}$ as $r\to\infty$,
thus $w(r;\b)\sim(K\log r)^{2/(2-p)}$, which is exactly the
logarithmic behavior expected when $\b\in C$. Thus, we are led to
the idea of showing that, for $\b\in C$, the inequality
\eqref{ineq.C} is in fact an equality for a suitable value of $K$.
This will be done by comparison. We first have the following upper bound which improves \eqref{ineq.C}.

\begin{lemma}\label{lemma.super}
Consider $\b\in C$. The following inequality holds true
\begin{equation}\label{upper.est}
\Phi(\xi)\leq
K(\b)\xi^{p/2}+\left(\frac{p\mu-N}{\mu-N}-K(\b)\right)_{+}\xi_0^{p/2}
\quad {\rm for} \ \xi>\xi_0,
\end{equation}
where
\begin{equation}\label{xi0}
K(\b):=\frac{\mu^{p/2}}{\b}, \quad
\xi_0:=\left[\max\left\{K(\b),\frac{p\mu-N}{\mu-N}\right\}\frac{p\mu-N}{\mu(\mu-N)}\right]^{2/(2-p)}.
\end{equation}
\end{lemma}

\begin{proof}
Let $\xi_0>0$ be given by \eqref{xi0}, $M>0$ to be determined later on, and define
$$
\Phi_{sup}(\xi):=K(\b)\xi^{p/2}+M\xi_{0}^{p/2}, \quad \xi\ge \xi_0.
$$
Then, for $\xi>\xi_0$, we have $\mu\xi \ge \Phi_{sup}(\xi)$ by \eqref{xi0} and
\begin{equation*}
\begin{split}
\cl\Phi_{sup}(\xi)&=\frac{p(p-1)}{2}\left(K(\b)\xi^{p/2}+M\xi_0^{p/2}\right)K(\b)\xi^{(p-2)/2}\\
&-(p\mu-N)\left(K(\b)\xi^{p/2}+M\xi_0^{p/2}\right)+\mu(\mu-N)\xi
+\left[\mu\xi-K(\b)\xi^{p/2}-M\xi_0^{p/2}\right]^{2-p}\\&\times\left[\b\left(K(\b)\xi^{p/2}+
M\xi_0^{p/2}\right)-\left[\mu\xi-K(\b)\xi^{p/2}-M\xi_0^{p/2}\right]^{p/2}\right]\\
&\geq\xi\left[\mu(\mu-N)-(p\mu-N)K(\b)\xi^{(p-2)/2}-(p\mu-N)M\xi_0^{p/2}\xi^{-1}\right]\\
&+\xi^{(4-p)/2}\left[\mu-K(\b)\xi^{(p-2)/2}-M\xi_0^{p/2}\xi^{-1}\right]^{2-p}\\&\times\left[\b
K(\b)+\b M\left(\frac{\xi_0}{\xi}\right)^{p/2}-\left[\mu-K(\b)\xi^{(p-2)/2}-M\xi_0^{p/2}\xi^{-1}\right]^{p/2}\right]
\end{split}
\end{equation*}
On the one hand, since $p\in(1,2)$ and $\xi\geq\xi_0$, we have
$\xi^{(p-2)/2}\leq\xi_0^{(p-2)/2}$ and $\xi^{-1}\leq\xi_0^{-1}$,
whence
\begin{equation*}
\begin{split}
\Big[\mu(\mu-N)\Big.&\Big.-(p\mu-N)K(\b)\xi^{(p-2)/2}-(p\mu-N)M\xi_0^{p/2}\xi^{-1}\Big]\\
&\geq\mu(\mu-N)-(p\mu-N)K(\b)\xi_0^{(p-2)/2}-(p\mu-N)M\xi_0^{(p-2)/2}\\&\geq\mu(\mu-N)-(p\mu-N)(K(\b)+M)\xi_0^{(p-2)/2},
\end{split}
\end{equation*}
and
\begin{equation*}
\begin{split}
\mu-K(\b)\xi^{(p-2)/2}-M\xi_0^{p/2}\xi^{-1}&\geq\mu-K(\b)\xi_0^{(p-2)/2}-M\xi_0^{(p-2)/2}\\
&=\mu-(K(\b)+M)\xi_0^{(p-2)/2}.
\end{split}
\end{equation*}
On the other hand, dropping some terms, we have
\begin{equation*}
\b\left[K(\b)+M\left(\frac{\xi_0}{\xi}\right)^{p/2}\right]-\left[\mu-K(\b)\xi^{(p-2)/2}-M\xi_0^{p/2}\xi^{-1}\right]^{p/2}\geq\b
K(\b)-\mu^{p/2}=0.
\end{equation*}
Choosing $M$ such that
\begin{equation}
(K(\b)+M)\xi_0^{(p-2)/2}\leq\frac{\mu(\mu-N)}{p\mu-N}<\mu, \label{zz7}
\end{equation}
we end up with $\cl\Phi_{sup}(\xi)\geq 0$ for $\xi\geq\xi_0$. In
addition,
$$
\Phi_{sup}(\xi_0)=(K(\b)+M)\xi_0^{p/2}\geq\max\left\{K(\b),\frac{p\mu-N}{\mu-N}\right\}\xi_0^{p/2}\geq\Phi(\xi_0),
$$
by Lemma~\ref{lemma.boundC}, provided
\begin{equation}
M\geq\max\left\{K(\b),\frac{p\mu-N}{\mu-N}\right\}-K(\b)=\left(\frac{p\mu-N}{\mu-N}-K(\b)\right)_{+}. \label{zz8}
\end{equation}
Taking $M=((p\mu-N)/(\mu-N)-K(\b))_{+}$, the choice \eqref{xi0} of $\xi_0$ guarantees that the conditions \eqref{zz7} and \eqref{zz8} are satisfied, so that we may apply the comparison principle and obtain the claimed upper bound.
\end{proof}

We now establish an optimal lower bound for $\Phi(\cdot;\beta)$.

\begin{lemma}\label{lemma.lower}
Consider $\b\in C$. Given $\e\in(0,K(\b))$ with $K(\b)$ defined in
\eqref{xi0}, there exists $\xi_{\e}>0$ such that
\begin{equation}\label{lowerC}
\Phi(\xi;\beta)\geq(K(\b)-\e)(\xi^{p/2}-\xi_{\e}^{p/2}), \quad
\xi>\xi_{\e}.
\end{equation}
\end{lemma}

\begin{proof}
Consider $\e\in(0,K(\b))$ and
$\Phi_{sub}(\xi):=(K(\b)-\e)(\xi^{p/2}-\xi_{\e}^{p/2})$ for
$\xi\in(\xi_{\e},\infty)$, where $\xi_{\e}$ is to be determined.
We first notice that, for $\xi\ge \xi_\e$,
\begin{equation}
\begin{split}
\mu\xi - \Phi_{sub}(\xi) &\geq\xi\left(\mu-(K(\b)-\e)\xi^{(p-2)/2}\right)\\
&\geq\xi\left(\mu-K(\b)\xi_{\e}^{(p-2)/2}\right)\geq\frac{\mu}{2}\xi\geq0,
\end{split}\label{zz9}
\end{equation}
provided that
\begin{equation}\label{part10}
\xi_{\e}^{(p-2)/2}\leq \frac{\mu}{2K(\b)}.
\end{equation}
Then, for $\xi>\xi_{\e}$, we have $\mu\xi\ge \Phi_{sub}(\xi)>0$ and
\begin{equation*}
\begin{split}
\cl\Phi_{sub}(\xi)&=\frac{p(p-1)}{2}(K(\b)-\e) \Phi_{sub}(\xi) \xi^{(p-2)/2}\\
&-(p\mu-N) \Phi_{sub}(\xi) +\mu(\mu-N)\xi\\&+\left[\mu\xi-(K(\b)-\e)\xi^{p/2}+(K(\b)-\e)\xi_{\e}^{p/2}\right]^{2-p}\\
&\times\Big\{ \b(K(\b)-\e)\xi^{p/2}-\b(K(\b)-\e)\xi_{\e}^{p/2}\Big.\\&\Big.-\left[\mu\xi-(K(\b)-\e)\xi^{p/2}+(K(\b)-\e)\xi_{\e}^{p/2}\right]^{p/2}\Big\}\\
&\leq\Phi_{sub}(\xi)\left[\frac{p(p-1)}{2}K(\b)\xi_{\e}^{(p-2)/2}-(p\mu-N)\right]+\mu(\mu-N)\xi\\
&+\left[\mu\xi-(K(\b)-\e)\xi^{p/2}+(K(\b)-\e)\xi_{\e}^{p/2}\right]^{2-p}\\
&\times\left\{ \b(K(\b)-\e)\xi^{p/2} -\left[\mu\xi-(K(\b)-\e)\xi^{p/2}\right]^{p/2} \right\}.
\end{split}
\end{equation*}
Now, since $\xi\geq\xi_{\e}$, we have
$\xi^{(p-2)/2}\leq\xi_{\e}^{(p-2)/2}$ and we can use the sublinearity of $z\mapsto z^{p/2}$ to estimate
\begin{equation*}
\begin{split}
&\b(K(\b)-\e)\xi^{p/2}-\left(\mu\xi-(K(\b)-\e)\xi^{p/2}\right)^{p/2}\\
&\leq\xi^{p/2}\left[\b(K(\b)-\e)-\left(\mu-(K(\b)-\e)\xi^{(p-2)/2}\right)^{p/2}\right]\\
&\leq\xi^{p/2}\left[\mu^{p/2}-\b\e-\left(\mu-(K(\b)-\e)\xi_{\e}^{(p-2)/2}\right)^{p/2}\right]\\
&\leq\xi^{p/2}\left[(K(\b)-\e)^{p/2}\xi_{\e}^{(p-2)p/4}-\b \e\right]\\
&\leq-\frac{\b\e}{2}\xi^{p/2},
\end{split}
\end{equation*}
provided
\begin{equation}\label{part8}
(K(\b)-\e)\xi_{\e}^{(p-2)/2} \leq K(\b) \xi_{\e}^{(p-2)/2} \leq \left(\frac{\b\e}{2}\right)^{2/p}.
\end{equation}
Therefore, for $\xi\geq\xi_{\e}$, it follows from \eqref{zz9} that
\begin{equation*}
\begin{split}
\cl\Phi_{sub}(\xi)&\leq\Phi_{sub}(\xi)\frac{p(p-1)}{2}(K(\b)-\e)\left(\xi_{\e}^{(p-2)/2}-\frac{2(p\mu-N)}{p(p-1)K(\b)}\right)\\
&+\mu(\mu-N)\xi-\frac{\b\e}{2}\left(\frac{\mu}{2}\right)^{2-p}\xi^{(4-p)/2}\\
&\leq\Phi_{sub}(\xi)\frac{p(p-1)(K(\b)-\e)}{2}\left(\xi_{\e}^{(p-2)/2}-\frac{2(p\mu-N)}{p(p-1)K(\b)}\right)\\
&+\mu(\mu-N)\xi^{(4-p)/2}\left(\xi^{(p-2)/2}-\frac{\b\e}{2}\left(\frac{\mu}{2}\right)^{2-p}\frac{1}{\mu(\mu-N)}\right)\\
& \leq\Phi_{sub}(\xi)\frac{p(p-1)(K(\b)-\e)}{2}\left(\xi_{\e}^{(p-2)/2}-\frac{2(p\mu-N)}{p(p-1)K(\b)}\right)\\
& +\mu(\mu-N)\xi^{(4-p)/2}\left(\xi_\e^{(p-2)/2}-\frac{\b\e}{2}\left(\frac{\mu}{2}\right)^{2-p}\frac{1}{\mu(\mu-N)}\right) \leq 0,
\end{split}
\end{equation*}
provided that
\begin{equation}\label{part9}
\xi_{\e}^{(p-2)/2}\leq\frac{2(p\mu-N)}{p(p-1)K(\b)} \quad\mbox{ and }\quad
\xi_{\e}^{(p-2)/2}\leq\frac{\b\e}{2}\left(\frac{\mu}{2}\right)^{2-p}\frac{1}{\mu(\mu-N)}.
\end{equation}
So, if we set
\begin{equation}
\xi_{\e}^{(p-2)/2} := \min\left\{\frac{2(p\mu-N)}{p(p-1)K(\b)},\frac{\mu}{2K(\b)},\left(\frac{\b}{2}\right)^{2/p}\frac{\e^{2/p}}{K(\b)},
\frac{\b\e}{2\mu(\mu-N)}\left(\frac{\mu}{2}\right)^{2-p}\right\},\label{zz10}
\end{equation}
the conditions \eqref{part10}, \eqref{part8}, and \eqref{part9} are clearly
satisfied. Then, for $\xi_\e$ given by \eqref{zz10}, we have $\mathcal{L} \Phi_{sub}(\xi)\le 0$ for $\xi\ge \xi_\e$. In addition, $\Phi_{sub}(\xi_\e)=0<\Phi(\xi_\e)$ by \eqref{estimate.Phi1} and the comparison principle gives \eqref{lowerC}.
\end{proof}

Combining the outcome of Lemma~\ref{lemma.super} and Lemma~\ref{lemma.lower}, we may now identify the behavior of $\Phi(\xi;\beta)$ as $\xi\to\infty$ for $\beta\in C$.

\begin{corollary}\label{cor.C}
For $\b\in C$, we have
$$
\lim\limits_{\xi\to\infty}\frac{\Phi(\xi)}{\xi^{p/2}}=K(\b) = \frac{\mu^{p/2}}{\beta}.
$$
\end{corollary}

\begin{proof}
Let $\e\in(0,K(\b))$. Then, for $\xi>\max\{\xi_0,\xi_{\e}\}$ we have
$$
(K(\b)-\e)\left( 1- \left( \frac{\xi_{\e}}{\xi} \right)^{p/2} \right)\leq\frac{\Phi(\xi)}{\xi^{p/2}}\leq
K(\b) + \left( \frac{p\mu-N}{\mu (\mu-N)} - K(\beta) \right)_+\ \left( \frac{\xi_0}{\xi} \right)^{p/2},
$$
by \eqref{upper.est} and \eqref{lowerC}. Therefore,
$$
K(\b)-\e \leq\liminf\limits_{\xi\to\infty}\frac{\Phi(\xi)}{\xi^{p/2}}\leq\limsup\limits_{\xi\to\infty}\frac{\Phi(\xi)}{\xi^{p/2}}=K(\b).
$$
Since the last inequalities are valid for all $\e\in(0,K(\b))$, the
conclusion follows by letting $\e\to0$.
\end{proof}

\subsection{Behavior as $\xi\to w^*$ for $\b\in B$}\label{seabB}

We turn now our attention to the case $\b\in B$. Then $\xi^*(\beta)=w^*$ and we first prove the following
preliminary result.

\begin{lemma}\label{lemma.limitB}
If $\b\in B$, then
\begin{equation}\label{limB}
\lim\limits_{\xi\to w^{*}}\Phi(\xi;\b)=\lim\limits_{r\to\infty}rw'(r;\b)=0.
\end{equation}
\end{lemma}

\begin{proof}
Set $\Phi=\Phi(\cdot;\beta)$. By Lemma~\ref{lemma.estPhi}, we have $0<\Phi(\xi)<\mu\xi<\mu w^*$
and it follows from \eqref{equation.Phi} that
\begin{equation*}
\frac{p-1}{2}(\Phi^2)'(\xi)=\left[(p\mu-N)-\b(\mu\xi-\Phi(\xi))^{2-p}\right]\Phi(\xi)-\mu(\mu-N)\xi+(\mu\xi-\Phi(\xi))^{(4-p)/2},
\end{equation*}
whence
$$
\frac{p-1}{2}\left|(\Phi^2)'(\xi)\right|\leq\left[p\mu-N+\b(\mu
w^*)^{2-p}\right]\mu w^*+\mu(\mu-N)w^*+(\mu w^*)^{(4-p)/2}.
$$
Consequently, $(\Phi^2)'\in L^{1}(0,w^*)$ and
$$
\Phi(\xi)^2=\int_0^{\xi}(\Phi^2)'(\eta)\,d\eta
$$
has a limit as $\xi\to w^*$. This readily implies that $\Phi$ has a limit as
$\xi\to w^*$, which is denoted by $\Phi(\xi^*)$ and is
nonnegative by \eqref{estimate.Phi1}. Coming back to \eqref{def.Phi}, this fact ensures that
$rw'(r;\b)\to\Phi(\xi^*)$ as $r\to\infty$. Since the properties $w(r;\beta)\to w^*$ as $r\to \infty$ and $w'(\cdot;\beta)>0$ imply that $w'(\cdot;\beta)$ belongs to $L^1(0,\infty)$, we necessarily have
$\Phi(\xi^*)=0$ as claimed.
\end{proof}

\medskip

\noindent \textbf{Another formulation for \eqref{equation.Phi}. Consider $\beta\in B$.} Using the definition \eqref{zw} of $w^*$, we write the equation \eqref{equation.Phi} as follows:
\begin{equation*}
\begin{split}
(p-1)(\Phi\Phi')(\xi;\beta)&=(p\mu-N)\Phi(\xi;\beta)-\b\Phi(\xi;\beta)(\mu\xi-\Phi(\xi;\beta))^{2-p}\\&+(\mu\xi-\Phi(\xi;\beta))^{(4-p)/2}-\mu\xi(\mu w^*)^{(2-p)/2}\\
&=\left[(p\mu-N)-\b(\mu\xi-\Phi(\xi;\beta))^{2-p}\right]\Phi(\xi;\beta)\\
&+(\mu\xi-\Phi(\xi;\beta))(\mu\xi-\Phi(\xi;\beta))^{(2-p)/2}-\mu\xi(\mu w^*)^{(2-p)/2}\\
&=\left[(p\mu-N)-\b(\mu\xi-\Phi(\xi;\beta))^{2-p}-(\mu\xi-\Phi(\xi;\beta))^{(2-p)/2}\right]\Phi(\xi;\beta)\\
&+\mu\xi\left[(\mu w^*-\mu(w^*-\xi)-\Phi(\xi;\beta))^{(2-p)/2}-(\mu w^*)^{(2-p)/2}\right].
\end{split}
\end{equation*}
We introduce
\begin{equation}\label{alt1}
\mathcal{A}(\xi;\beta):=(p\mu-N)-\b(\mu\xi-\Phi(\xi;\beta))^{2-p}-(\mu\xi-\Phi(\xi;\beta))^{(2-p)/2}
\end{equation}
and
\begin{equation}\label{alt2}
\mathcal{B}(\xi;\beta):=-\frac{\left(\mu
w^*-\mu(w^*-\xi)-\Phi(\xi;\beta)\right)^{(2-p)/2}-(\mu
w^*)^{(2-p)/2}}{\mu(w^*-\xi)+\Phi(\xi;\beta)}
\end{equation}
for $\xi\in[0,w^*)$. We notice that Lemma~\ref{lemma.limitB} implies
that
\begin{equation}\label{limit1}
\lim\limits_{\xi\to w^*} \mathcal{A}(\xi;\beta)=(p\mu-N)-\b(\mu
w^*)^{2-p}-(\mu w^*)^{(2-p)/2} = (p-1)\mu - \beta (\mu-N)^2
\end{equation}
and
\begin{equation}\label{limit2}
\lim\limits_{\xi\to w^*} \mathcal{B}(\xi;\beta)=\frac{2-p}{2}(\mu w^*)^{-p/2}.
\end{equation}
Then, for $\xi\in (0,w^*)$,
\begin{equation*}
\begin{split}
(p-1)(\Phi\Phi')(\xi;\beta)&=\mathcal{A}(\xi;\beta)\Phi(\xi;\beta)-\mu\xi
\mathcal{B}(\xi;\beta)(\mu(w^*-\xi)+\Phi(\xi;\beta))\\&=(\mathcal{A}(\xi;\beta)-\mu\xi
\mathcal{B}(\xi;\beta))\Phi(\xi;\beta)-\mu^2\xi \mathcal{B}(\xi;\beta)(w^*-\xi),
\end{split}
\end{equation*}
hence we can write
$$
2(\Phi\Phi')(\xi;\beta)=a(\xi;\beta)\Phi(\xi;\beta)-2b(\xi;\beta)(w^*-\xi),
$$
where
\begin{equation}\label{alt3}
a(\xi;\beta):=\frac{2\left[\mathcal{A}(\xi;\beta)-\mu\xi \mathcal{B}(\xi;\beta)\right]}{p-1}, \quad
b(\xi;\beta):=\frac{\mu^2\xi \mathcal{B}(\xi;\beta)}{p-1}.
\end{equation}
Introducing $\Psi(\xi;\beta):=\Phi(w^*-\xi;\beta)$ for $\xi\in [0,w^*]$, we end up with the following
alternative formulation of \eqref{equation.Phi}:
\begin{equation}\label{equation.Psi}
2\Psi(\xi;\beta)\Psi'(\xi;\beta)+a(w^*-\xi;\beta)\Psi(\xi;\beta)-2b(w^*-\xi;\beta)\xi=0
\end{equation}
with initial condition $\Psi(0;\beta)=0$. Observe that it follows from \eqref{limit1} and \eqref{limit2} that
\begin{equation}\label{limit3}
\lim\limits_{\xi\to0}a(w^*-\xi;\beta)=a^*(\beta) \in\real \quad\mbox{ and }\quad
\lim\limits_{\xi\to0}b(w^*-\xi;\beta)=b^*(\beta)>0.
\end{equation}
With the help of this alternative form, we can study the behavior of
$\Phi(\xi;\beta)$ as $\xi\to w^*$. More precisely:

\begin{lemma}\label{lemma.estinfB}
Let $\b\in B$. There exists a constant $K^*(\beta)>0$ such that
\begin{equation}\label{lim.B}
\lim\limits_{\xi\to w^*}\frac{\Phi(\xi;\beta)}{w^*-\xi}=K^*(\beta).
\end{equation}
\end{lemma}

\begin{proof}
To simplify notation, we omit the $\beta$-dependence of $\Psi$, $K^*$, $a^*$, and $b^*$ in the proof. We use comparison with suitable subsolutions and supersolutions. Fix $\e\in(0,b^*)$ with $b^*$ introduced in \eqref{limit3}. Then, there
exists $\Xi_{\e}>0$ such that, for $\xi\in(0,\Xi_{\e})$,
\begin{equation}\label{estimate.alt}
a^*-\e\leq a(w^*-\xi)\leq a^*+\e, \quad b^*-\e\leq b(w^*-\e)\leq
b^*+\e.
\end{equation}
Consider $\delta>0$ and define $$\Psi_{sup}(\xi):=\delta+M\xi, \quad
\xi\in(0,\Xi_{\e}),$$ for some $M$ to be determined later. Then, for
$\xi\in(0,\Xi_{\e})$, we infer from \eqref{estimate.alt} that
\begin{equation*}
\begin{split}
2\Psi_{sup}(\xi)\Psi_{sup}'(\xi)&+a(w^*-\xi)\Psi_{sup}(\xi)-2b(w^*-\xi)\xi\\
&\geq2M(\delta+M\xi)+(a^*-\e)(\delta+M\xi)-2(b^*+\e)\xi\\
&\geq\left(2M^2+(a^*-\e)M-2(b^*+\e)\right)\xi+(2M+a^*-\e)\delta.
\end{split}
\end{equation*}
Choosing
$$
M=K_{\e}:=\frac{-(a^*-\e)+\sqrt{(a^*-\e)^2+16(b^*+\e)}}{4}>0,
$$
we note that $2K_{\e}^2+(a^*-\e)K_{\e}=2(b^*+\e) > 0$, so that also
$2K_{\e}+a^*-\e > 0$ and thus
\begin{equation}\label{part11}
2\Psi_{sup}(\xi)\Psi_{sup}'(\xi)+a(w^*-\xi)\Psi_{sup}(\xi)-2b(w^*-\xi)\xi>0
\end{equation}
for $\xi\in(0,\Xi_{\e})$. Since $\Psi(0)=0<\delta=\Psi_{sup}(0)$, we
have $\Psi<\Psi_{sup}$ in a right neighborhood of $\xi=0$, hence
$$
\overline{\xi}:=\inf\{\xi\in[0,\Xi_{\e}]:
\Psi(\xi)\geq\Psi_{sup}(\xi)\}>0.
$$
Assume for contradiction that $\overline{\xi}<\Xi_{\e}$. Then
$\Psi(\overline{\xi})=\Psi_{sup}(\overline{\xi})>0$ and
$\Psi(\xi)<\Psi_{sup}(\xi)$ for $\xi\in[0,\overline{\xi})$, so that
$\Psi'(\overline{\xi})\geq\Psi_{sup}'(\overline{\xi})$.  Moreover, \eqref{equation.Psi} and \eqref{part11} for
$\xi=\overline{\xi}$ imply
\begin{equation*}
\begin{split}
2\Psi(\overline{\xi})&\Psi_{sup}'(\overline{\xi})+a(w^*-\overline{\xi})\Psi(\overline{\xi})-2b(w^*-\overline{\xi})\overline{\xi}\\
&>0=2\Psi(\overline{\xi})\Psi'(\overline{\xi})+a(w^*-\overline{\xi})\Psi(\overline{\xi})-2b(w^*-\overline{\xi})\overline{\xi},
\end{split}
\end{equation*}
whence $\Psi_{sup}'(\overline{\xi})>\Psi'(\overline{\xi})$ and a
contradiction. We have thus shown that $\Psi(\xi)\leq\Psi_{sup}(\xi)$
for all $\xi\in[0,\Xi_{\e}]$, hence $\Psi(\xi)\leq\delta+K_{\e}\xi$,
for all $\xi\in[0,\Xi_{\e}]$. The above upper bound being true for
any $\delta>0$, we conclude that
\begin{equation}\label{upper.Psi}
\Psi(\xi)\leq K_{\e}\xi, \quad \hbox{for} \ \xi\in[0,\Xi_{\e}].
\end{equation}

\medskip

In order to obtain a similar lower bound, we next consider
$\delta\in(0,\Xi_{\e})$ and define
$$
\Psi_{sub}(\xi):=L(\xi-\delta), \quad \xi\in(\delta,\Xi_{\e})
$$
for some $L$ to be determined later. It then follows from
\eqref{limit3} that, for $\xi\in(\delta,\Xi_{\e})$, we have
\begin{equation*}
\begin{split}
2\Psi_{sub}(\xi)\Psi_{sub}'(\xi)&+a(w^*-\xi)\Psi_{sub}(\xi)-2b(w^*-\xi)\xi\\
&\leq2L^2(\xi-\delta)+L(a^*+\e)(\xi-\delta)-2(b^*-\e)\xi\\
&\leq\left(2L^2+(a^*+\e)L-2(b^*-\e)\right)\xi-L(2L+(a^*+\e))\delta.
\end{split}
\end{equation*}
Choosing
$$
L=L_{\e}:=\frac{-(a^*+\e)+\sqrt{(a^*+\e)^2+16(b^*-\e)}}{4},
$$
we note that $2L_{\e}^2+(a^*+\e)L_{\e}=2(b^*-\e)>0$, hence
\begin{equation}\label{part12}
2\Psi_{sub}(\xi)\Psi_{sub}'(\xi)+a(w^*-\xi)\Psi_{sub}(\xi)-2b(w^*-\xi)\xi<0,
\end{equation}
for $\xi\in(\delta,\Xi_{\e})$. Now,
$\Psi_{sub}(\delta)=0<\Psi(\delta)$ and, using \eqref{part12}, we argue as above by contradiction to show that
$$
L_{\e}(\xi-\delta)\leq\Psi(\xi), \quad \hbox{for} \
\xi\in[\delta,\Xi_{\e}].
$$
This lower bound being valid for any $\delta\in(0,\Xi_{\e})$, we
conclude that
\begin{equation}
L_{\e}\xi\leq\Psi(\xi), \quad \xi\in[0,\Xi_{\e}].\label{zz11}
\end{equation}
Observing that
$$
\lim\limits_{\e\to 0}L_{\e}=\lim\limits_{\e\to0}K_{\e}=K^*:=\frac{\sqrt{(a^*)^2+16b^*}-a^*}{4}>0,
$$
we infer from \eqref{upper.Psi} and \eqref{zz11} that $\Psi(\xi)/\xi\to K^*$ as $\xi\to 0$, from which the conclusion
follows.
\end{proof}

The last step needed for the proof of Theorem~\ref{th1}~(i) is related to the behavior of $\partial_\beta\Phi(\xi;\beta)$ as $\xi\to w^*$ when $\b$ lies in the interior of $B=[\b_*,\b^*]$ (if it is non-empty).

\begin{lemma}\label{lemma.estderiv}
Assume that $\beta_*<\beta^*$. Then
$$
(\xi,\beta)\longmapsto \Phi(\xi;\beta) \;\;\mbox{ belongs to }\;\; C([0,w^*)\times (\beta_*,\beta^*))\cap C^1((0,w^*)\times (\beta_*,\beta^*))
$$
and, for $\beta\in (\beta_*,\beta^*)$, $\partial_\beta\Phi(0;\beta)=0$,
\begin{equation}\label{estderiv}
\partial_\beta\Phi(\xi;\beta) \le 0 \;\;\mbox{ for }\;\; \xi\in (0,w^*) \;\;\mbox{ and }\;\; \lim_{\xi\to w^*} \partial_\beta\Phi(\xi;\beta) = - \infty\,.
\end{equation}
\end{lemma}

\begin{proof}
First, the regularity of $\Phi$ and the property $\partial_\beta\Phi(0;\beta)=0$ for $\beta\in (\beta_*,\beta^*)$ follow from the regularity of $w$ and its monotonicity with respect to $r$ by \eqref{def.Phi} and the implicit function theorem, while the non-positivity of $\partial_\beta\Phi(\cdot;\beta)$ is a consequence of Lemma~\ref{lemma.monotPhi}. Next, we fix $\beta\in (\beta_*,\beta^*)$ and set $\Phi=\Phi(\cdot;\beta)$ and $\Phi_{\b}=\partial_{\b}\Phi(\cdot;\b)$. Dividing \eqref{equation.Phi} by $\Phi$ gives, for $\xi\in (0,w^*)$,
\begin{equation}\label{alt4}
(p-1)\Phi'(\xi)+(N-p\mu)+\frac{\mu(\mu-N)\xi}{\Phi(\xi)}+\b(\mu\xi-\Phi(\xi))^{2-p}-\frac{(\mu\xi-\Phi(\xi))^{(4-p)/2}}{\Phi(\xi)}=0.
\end{equation}
We differentiate the equation \eqref{alt4} with respect to $\b$ to
obtain
\begin{equation}\label{alt5}
(p-1)\Phi_{\b}'(\xi)+T(\xi)\Phi_{\b}(\xi)+(\mu\xi-\Phi(\xi))^{2-p}=0, \quad \xi\in (0,w^*),
\end{equation}
with
\begin{equation*}
\begin{split}
T(\xi)&:=\frac{\left(\mu\xi-\Phi(\xi)\right)^{(4-p)/2}-\mu\xi(\mu-N)}{\Phi(\xi)^2}+\frac{4-p}{2}\frac{\left(\mu\xi-\Phi(\xi)\right)^{(2-p)/2}}{\Phi(\xi)}\\
&-\b(2-p)(\mu\xi-\Phi(\xi))^{1-p}.
\end{split}
\end{equation*}
We now estimate $T(\xi)$ as $\xi\to w^*$: it follows from \eqref{zw} that $T=T_1+T_2+T_3$ with
\begin{eqnarray*}
T_1(\xi) & := & \mu\xi\frac{\left(\mu\xi-\Phi(\xi)\right)^{(2-p)/2}-(\mu
w^*)^{(2-p)/2}}{\Phi(\xi)^2}\\
T_2(\xi) & := & \frac{2-p}{2}\frac{\left(\mu\xi-\Phi(\xi)\right)^{(2-p)/2}}{\Phi(\xi)}\,, \quad T_3(\xi) := -\b(2-p)(\mu\xi-\Phi(\xi))^{1-p}\,.
\end{eqnarray*}
Owing to \eqref{lim.B}, we have as $\xi\to w^*$,
\begin{eqnarray*}
T_1(\xi) & = & \frac{\mu\xi}{\Phi(\xi)^2}\ \left[ \left( \mu w^* - \mu (w^*-\xi) - K^* (w^*-\xi) + o(w^*-\xi) \right)^{(2-p)/2} - (\mu w^*)^{(2-p)/2} \right] \\
& = & \frac{\mu\xi}{\Phi(\xi)^2}\ (\mu w^*)^{(2-p)/2}\ \left[ \left( 1 - \frac{\mu+K^*}{\mu w^*} (w^*-\xi) + o(w^*-\xi) \right)^{(2-p)/2} - 1 \right] \\
&  \sim & - \frac{2-p}{2}\ \frac{(\mu w^*)^{(2-p)/2}}{(K^*)^2}\ \frac{\mu+K^*}{w^*-\xi}\,,
\end{eqnarray*}
and
\begin{equation*}
T_2(\xi) \sim \frac{2-p}{2} \frac{(\mu w^*)^{(2-p)/2}}{K^* (w^*-\xi)}\,, \qquad T_3(\xi) \sim - \beta(2-p) (\mu w^*)^{1-p}\,.
\end{equation*}
Therefore,
\begin{equation}\label{part13}
T(\xi)\sim - \frac{2(p-1)\kappa_0}{w^*-\xi} \quad \hbox{as} \ \xi\to w^* \;\;\mbox{ with }\;\; \kappa_0 := \frac{\mu p (\mu-N)}{4(p-1) K^*}>0.
\end{equation}

\medskip

Owing to \eqref{limB} and \eqref{part13}, there is $\delta_\beta\in (0,w^*)$ such that
\begin{eqnarray}
T(\xi) & \le & - \frac{(p-1)\kappa_0}{w^*-\xi}\,, \qquad \xi\in (w^*-\delta_\beta,w^*)\,, \label{zz30} \\
\frac{(\mu\xi - \Phi(\xi))^{2-p}}{p-1} & \ge & \kappa_1 := \frac{1}{p-1}\ \left( \frac{\mu w^*}{2} \right)^{2-p}\,, \qquad \xi\in (w^*-\delta_\beta,w^*)\,. \label{zz31}
\end{eqnarray}
Consider now the solution $Z\in C^1([w^*-\delta_\beta,w^*))$ to
\begin{equation}
Z'(\xi) - \frac{\kappa_0}{w^*-\xi}\ Z(\xi) + \kappa_1 = 0\,, \qquad \xi\in (w^*-\delta_\beta,w^*)\,, \label{zz32}
\end{equation}
with initial condition $Z(w^*-\delta_\beta)=0$. We infer from \eqref{alt5}, \eqref{zz30}, \eqref{zz31}, \eqref{zz32}, and the non-positivity of $\Phi_\beta$ that $\Phi_\beta(w^*-\delta_\beta)\le 0=Z(w^*-\delta_\beta)$ and, for $\xi\in (w^*-\delta_\beta,w^*)$,
\begin{eqnarray*}
\Phi_\beta'(\xi) - \frac{\kappa_0}{w^*-\xi}\ \Phi_\beta(\xi) + \kappa_1 & = & - \frac{\Phi_\beta(\xi)}{p-1}\ \left( T(\xi) + \frac{(p-1)\kappa_0}{w^*-\xi} \right) \\
& + & \kappa_1 -  \frac{(\mu\xi - \Phi(\xi))^{2-p}}{p-1} \le 0\,.
\end{eqnarray*}
The comparison principle then ensures that
\begin{equation}
\Phi_\beta(\xi) \le Z(\xi)\,, \qquad \xi\in [w^*-\delta_\beta,w^*)\,. \label{zz33}
\end{equation}
Since
$$
Z(\xi) = \frac{\kappa_1}{1+\kappa_0}\ \left( w^* - \xi - \frac{\delta_\beta^{1+\kappa_0}}{(w^*-\xi)^{\kappa_0}} \right)\,, \qquad \xi\in [w^*-\delta_\beta,w^*)\,,
$$
the function $Z$ clearly diverges to $-\infty$ as $\xi\to w^*$ and so does $\Phi_\beta$ by \eqref{zz33}.
\end{proof}

\subsection{Proof of Theorem~\ref{th1}}\label{seeop}

We are now ready to conclude the proof of Theorem~\ref{th1}.

\begin{proof}[Proof of Theorem~\ref{th1}~(i)] It only remains to prove that the set $B$ reduces to one element. Assume thus for contradiction that $\beta_*<\beta^*$. It follows from \eqref{limB}, Lemma~\ref{lemma.estderiv}, and Fatou's lemma that, given $\beta_*<\beta_1<\beta_2<\beta^*$,
\begin{eqnarray*}
0 = \lim_{\xi\to w^*} \left( \Phi(\xi;\beta_1) - \Phi(\xi;\beta_2) \right) & = & \lim_{\xi\to w^*} \int_{\beta_1}^{\beta_2} \left|\partial_\beta \Phi(\xi;\gamma) \right|\ d\gamma \\
& \ge & \int_{\beta_1}^{\beta_2} \liminf_{\xi\to w^*} \left|\partial_\beta \Phi(\xi;\gamma) \right|\ d\gamma = \infty\,,
\end{eqnarray*}
and a contradiction. Then, $B$ reduces to one single point $B=\{\b_*\}$ and Theorem~\ref{th1}~(i) is proved.
\end{proof}

\begin{proof}[Proof of Theorem~\ref{th1}~(ii)] Let $\b\in C$ and $\e\in (0,K(\beta))$.
Recalling the definition \eqref{def.Phi} of $\Phi$, it follows from Corollary~\ref{cor.zz} and Corollary~\ref{cor.C} that there is $r_\e\ge 1$ such that
$$
K(\b)-\e\leq\frac{rw'(r;\b)}{w(r;\b)^{p/2}}\leq K(\b) + \e, \qquad r\ge r_\e\,.
$$
Integrating the above inequalities with respect to $r$ gives
$$
(K(\b)-\e)\log{r} + \frac{2 w(r_\e;\b)^{(2-p)/2}}{2-p} - (K(\b)-\e)\log{r_\e} \le \frac{2 w(r;\b)^{(2-p)/2}}{2-p}
$$
and
$$
\frac{2 w(r;\b)^{(2-p)/2}}{2-p} \le (K(\b)+\e)\log{r} + \frac{2 w(r_\e;\b)^{(2-p)/2}}{2-p} - (K(\b)+\e)\log{r_\e}
$$
for $r\ge r_\e$. Consequently,
$$
\lim\limits_{r\to\infty} \frac{w(r;\b)^{(2-p)/2}}{\log{r}} =
\frac{2-p}{2}\ K(\b),
$$
and we get the conclusion of Theorem~\ref{th1}~(ii) since $\mu+1=2/(2-p)$.
\end{proof}

\section*{Acknowledgements}
Part of this work was done while the first author enjoyed the hospitality and
support of the Institut de Math\'ematiques de Toulouse, France.

\bibliographystyle{plain}


\end{document}